\documentclass{amsart}
\usepackage{mathtools,amssymb,amscd,url,gensymb,array,enumerate}
\usepackage[enableskew]{youngtab}
\usepackage[backend=bibtex]{biblatex}
\usepackage[utf8]{inputenc}
\usepackage[T1]{fontenc}
\usepackage{tikz}
\usetikzlibrary{arrows.meta,tikzmark,decorations.markings,fit,calc,shapes,matrix,math}
\usepackage{hyperref}

\newtheorem{thm}{Theorem}[section]
\newtheorem{defn}[thm]{Definition}
\newtheorem{prop}[thm]{Proposition}
\newtheorem{lemma}[thm]{Lemma}
\newtheorem{cor}[thm]{Corollary}
\newtheorem{rmk}[thm]{Remark}

\newtheorem*{goal}{The main problem}
\theoremstyle{definition}
\newtheorem{example}[thm]{Example}

\newcommand{\qbinom}[2]{\genfrac{[}{]}{0pt}{}{#1}{#2}_q}

\newcommand{\Dfn}[1]{\emph{\color{blue}#1}}
\newcommand{\fS}{\mathfrak{S}}
\newcommand{\fH}{\mathfrak{H}}
\newcommand{\fC}{\mathfrak{C}}
\newcommand{\ZZ}{\mathbb{Z}}
\newcommand{\RR}{\mathbb{R}}
\newcommand{\SL}{\mathrm{SL}}
\newcommand{\SO}{\mathrm{SO}}
\newcommand{\Sp}{\mathrm{Sp}}
\newcommand{\GL}{\mathrm{GL}}
\newcommand{\Spin}{\mathrm{Spin}}

\DeclareMathOperator{\wt}{wt}
\newcommand{\wgt}{\mu}
\newcommand{\wgtosc}{\omega}
\newcommand{\osc}{\mathcal{O}} % an oscillating tableau
\newcommand{\alt}{\mathcal{A}} % an alternating tableau
\newcommand{\M}{\mathcal{M}} % the matching corresponding to an
\newcommand{\Perm}{\mathcal{P}} % the permutation corresponding to an
\newcommand{\G}{\mathcal{G}} % a growth diagram

\DeclareMathOperator{\rot}{rot} % rotation of a matching
\DeclareMathOperator{\pr}{pr} % promotion of an oscillating tableau
\newcommand{\half}[1]{\acute{#1}} % half a promotion
\DeclareMathOperator{\halfpr}{\half{pr}} % half a promotion,
\def\hm{\half{\wgt}}
\def\Pm{\hat{\wgt}}
\DeclareMathOperator{\ev}{ev} % evacuation
\DeclareMathOperator{\rev}{rev} % reversal
\DeclareMathOperator{\reco}{rc} % reverse-complement
\DeclareMathOperator{\dom}{dom} % get the dominant weight
\DeclareMathOperator{\domS}{sort} % dominant weight for Sn (sorting)
\DeclareMathOperator{\cactus}{s} % a generator of the cactus group
\newcommand{\lusztig}{\eta} % Lusztig's involution
\DeclareMathOperator{\jdt}{jdt} % jdt of a tableau
\DeclareMathOperator{\rrot}{rrot} % row rotation
\DeclareMathOperator{\crot}{crot} % column rotation
\newcommand{\pos}[1]{#1_+} % the positive part of a staircase
\renewcommand{\neg}[1]{#1_-} % the negative part of a staircase
\DeclareMathOperator{\extent}{\mathcal{E}} % extent of staircase
\newcommand*\circled[1]{\tikz[baseline=(char.base)]{
   \node[shape=circle,draw,inner sep=0.25pt,minimum size=3pt] (char) {#1};}}
\newcommand{\poscross}{{\circled{\small$+$}}}
\newcommand{\negcross}{{\circled{\small$-$}}}
\newcommand{\fixcross}{{\circled{\small$\times$}}}

\title[Promotion and rotation]{Promotion on oscillating and alternating tableaux and rotation of matchings and permutations}%

\author{Stephan Pfannerer}
\address[S.~Pfannerer and M.~Rubey]{Fakultät für Mathematik und Geoinformation, TU Wien, Austria}
\email{stephan.pfannerer@tuwien.ac.at, martin.rubey@tuwien.ac.at}
\author{Martin Rubey}
\thanks{Pfannerer and Rubey were supported by the the Austrian Science Fund (FWF): P 29275.}
\author{Bruce Westbury}
\address[B.~Westbury]{University of Texas at Dallas}
\email{bruce.westbury@gmail.com}

\begin{document}
\begin{abstract}
  Using Henriques' and Kamnitzer's cactus groups, Schützenberger's
  promotion and evacuation operators on standard Young tableaux can
  be generalised in a very natural way to operators acting on highest
  weight words in tensor products of crystals.

  For the crystals corresponding to the vector representations of the
  symplectic groups, we show that Sundaram's map to perfect matchings
  intertwines promotion and rotation of the associated chord
  diagrams, and evacuation and reversal.  We also exhibit a map with
  similar features for the crystals corresponding to the adjoint
  representations of the general linear groups.

  We prove these results by applying van Leeuwen's generalisation of
  Fomin's local rules for jeu de taquin, connected to the action of
  the cactus groups by Lenart, and variants of Fomin's growth
  diagrams for the Robinson-Schensted correspondence.%
\end{abstract}
\maketitle

\section{Introduction}
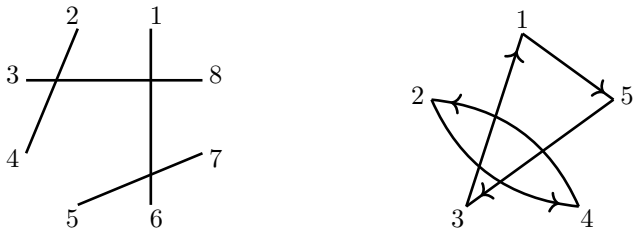
\begin{figure}[t]
  \begin{center}
  \begin{tikzpicture}[line width=1pt]
    \node (a) [draw=none, minimum size=2.5cm, regular polygon, regular polygon sides=8] at (0,0) {};
    \foreach \n [count=\i from 0, remember=\n as \lastn, evaluate={\i+\lastn}] in {1,2,...,8}
        \path (a.center) -- (a.corner \n) node[pos=1.15] {$\n$};
    \draw (a.corner 1) to (a.corner 6);
    \draw (a.corner 2) to (a.corner 4);
    \draw (a.corner 3) to (a.corner 8);
    \draw (a.corner 5) to (a.corner 7);
  \end{tikzpicture}
  \hfil
  \tikzset{->-/.style={decoration={
        markings,
        mark=at position #1 with {\arrow{>}}},postaction={decorate}}}
  \begin{tikzpicture}[line width=1pt]
    \node (a) [draw=none, minimum size=2.5cm, regular polygon, regular polygon sides=5] at (0,0) {};
    \foreach \n [count=\i from 0, remember=\n as \lastn, evaluate={\i+\lastn}] in {1,2,...,5}
        \path (a.center) -- (a.corner \n) node[pos=1.15] {$\n$};
    \draw[->-=.9] (a.corner 1) to (a.corner 5);
    \draw[->-=.9, bend right] (a.corner 2) to (a.corner 4);
    \draw[->-=.9] (a.corner 3) to (a.corner 1);
    \draw[->-=.9, bend right] (a.corner 4) to (a.corner 2);
    \draw[->-=.9] (a.corner 5) to (a.corner 3);
  \end{tikzpicture}
\end{center}
\caption{A $3$-noncrossing perfect matching and a permutation as
    chord diagrams.}
  \label{fig:chord-diagrams}
\end{figure}

This project began with the discovery that Sundaram's map from
perfect matchings, regarded as chord diagrams as in
Figure~\ref{fig:chord-diagrams}, to oscillating tableaux intertwines
rotation and promotion, see
Theorem~\ref{thm:rotation-promotion-matchings}.  Oscillating tableaux
are in bijection with highest weight words in a tensor power of the
crystal of the vector representation of the symplectic group
$\Sp(2n)$, and promotion is a natural generalisation of
Schützenberger's promotion map on standard Young tableaux.

We then found a map analogous to Sundaram's from permutations, again
regarded as chord diagrams, to Stembridge's alternating tableaux.
Alternating tableaux of length $r$ are in bijection with the highest
weight words in the $r$-th tensor power of the crystal for the
adjoint representation of the general linear group $\GL(n)$.  This
new map intertwines rotation and a suitable variant of promotion
provided that $n\geq r$, see
Theorem~\ref{thm:rotation-promotion-permutations}.  Finally, it
turned out that Theorem~\ref{thm:rotation-promotion-matchings} can be
deduced by a suitable embedding of the set of oscillating tableaux
into the set of alternating tableaux.

Both results are part of a more elaborate program, as we now explain.
A key observation is that Schützenberger's promotion and its above
mentioned variants can be understood in terms of an action of the
cactus groups.  These infinite groups, also known as quasi-braid
groups, were introduced by Devadoss~\cite[Def.~6.1.2]{MR1718078} and
placed into our context by Henriques and Kamnitzer~\cite{MR2219257}.
They defined a weight preserving action of the $r$-fruit cactus group
on highest weight words in $r$-fold tensor products of crystals.  The
action of a specific element of the cactus group generalises
promotion to an action on highest weight words in a tensor power of a
crystal.

As shown by Lenart~\cite{MR2361854}, the action of the cactus groups
can be made explicit using certain local rules discovered by van
Leeuwen~\cite{MR1661263}.  These rules generalise the classical local
rules for jeu de taquin by Fomin~\cite[App.~1]{MR1676282} to crystals
corresponding to any minuscule representation of a Lie group.  To
accommodate non-minuscule representations, one can use an embedding
into tensor products of minuscule representations, with a few
exceptions.

The link between highest weight words and diagrams involves classical
invariant theory.  Recall that the number of highest weight elements
of given weight in a crystal is the multiplicity of the irreducible
of the same weight in the direct sum decomposition of the
corresponding representation.  In particular, the number of highest
weight elements of weight zero is the dimension of the invariant
subspace.

The idea of using diagrams to index a basis of the invariant subspace
of a tensor power of a representation goes back to Rumer, Teller and
Weyl~\cite{zbMATH02549276}, and specifically Brauer~\cite{MR1503378}.
Given a perfect matching of $2r$ elements, he constructed an
invariant of the $2r$-th tensor power of the vector representation of
the symplectic group $\Sp(2n)$.  Furthermore, he showed that these
invariants linearly span the invariant space.  Using a result of
Sundaram~\cite{MR2941115}, Rubey and
Westbury~\cite{RubeyWestbury,RubeyWestbury2015} have shown that the
invariants obtained from perfect matchings without $(n+1)$-crossings
(see Section~\ref{sec:symplectic-results} for the definition) form a
basis of this space.

The symmetric group acts on a tensor power of a representation by
permuting tensor positions.  This action commutes with the action of
the Lie group.  In particular, the symmetric group also acts on the
invariant space of the tensor power.  

It is not hard to see that Brauer's construction translates rotation
of the chord diagram to the action of the long cycle of the symmetric
group on the corresponding invariant.  Moreover, it was shown by
Westbury~\cite{Westbury2009} that the action of the long cycle is
isomorphic to the action of promotion on highest weight words of
weight zero.

In general, given a representation, we would like to find a basis of
the invariant space indexed by diagrams, such that rotation of
diagrams corresponds to the action of the long cycle on the invariant
space.  We provide a review of such diagrammatic bases for invariant
spaces of tensor powers of other representations in
Section~\ref{sec:background}.

It then remains to establish an explicit bijection between the set of
diagrams and the set of highest weight words of weight zero which
intertwines rotation and promotion.  This is the focus of the present
article.

The initial motivation to study this problem arises from Reiner,
Stanton and White's cyclic sieving phenomenon~\cite{MR2087303}.
Essentially, this phenomenon occurs when the character of a cyclic
group action can be expressed as a polynomial in a particularly
simple way.  A standard example is the rotation action on noncrossing
perfect matchings of $\{1,\dots,2r\}$.  In this case, the $q$-Catalan
\lq number\rq\ $\frac{1}{[r+1]_q}\qbinom{2r}{r}$ is a cyclic sieving
polynomial: the evaluation at $q=e^{k\pi i/r}$ yields the number of
noncrossing perfect matchings invariant under rotation by $k$ points.

Consider the representation of the symmetric group on the invariant
space, and recall that the diagrammatic basis is preserved by the
action of the long cycle.  Then the cyclic sieving polynomial for
this cyclic action can be extracted from the Frobenius character of
the symmetric group action.  Although the Frobenius character of the
invariant spaces of tensor powers of representations is in general
hard to compute, it is known for several representations of interest,
in particular for the vector representation of $\Sp(2n)$ and the
adjoint representation of $\GL(n)$.

This article can therefore be regarded as an explicit demonstration
of the fact that the cyclic sieving phenomena for promotion of
highest weight words of weight zero and rotation of diagrams are the
same in the case of the vector representation of $\Sp(2n)$ and in the
case of the adjoint representation of $\GL(n)$.

In Section~\ref{sec:background} we recall some background material on
crystals of minuscule representations, make our goal precise and
indicate the conditions necessary to make our methods work.
Furthermore, we provide a summary of our contributions and what is
already known.  Precise statements of the new results are given in
Section~\ref{sec:results}.  The general machinery connecting the
action of the cactus groups on highest weight words, local rules, and
promotion is developed in Section~\ref{sec:crystals-local-rules}.  In
Section~\ref{sec:growth-diagrams} we introduce the two fundamental
growth diagram bijections, which apply only to oscillating and
alternating tableaux, respectively.  The proofs that these bijections
indeed intertwine promotion and rotation are delivered in the final
section, along with some additional material.
% \subsection*{Acknowledgements}
% We are very grateful to two anonymous referees, whose comments led to
% significant improvements in the exposition.

\section{Crystals and highest weight words}
\label{sec:background}
In this section, before stating the main goal of this article
precisely, we provide some background information on minuscule
representations, crystals and their tensor products.  We also recall
explicit combinatorial realisations of the associated highest weight
words, and survey in which cases partial solutions to the questions
mentioned in the introduction are known.  Although we will
subsequently only consider the adjoint representation of $\GL(n)$ and
the vector representation of $\Sp(2n)$, we provide the background in
a more general setting, to place our results into a bigger
picture.

A representation of a Lie group is \Dfn{minuscule} if its Weyl group
$W$ acts transitively on the weights of the representation: the set
of weights forms a single orbit under the action of $W$.  The
non-trivial minuscule representations are:
\begin{description}\setlength{\itemsep}{0pt}
\item[Type $A_n$] All exterior powers of the vector representation.
\item[Type $B_n$] The spin representation.
\item[Type $C_n$] The vector representation.
\item[Type $D_n$] The vector representation and the two half-spin
  representations.
\item[Type $E_6$] The two fundamental representations of dimension
  $27$.
\item[Type $E_7$] The fundamental representation of dimension $56$.
\end{description}
There are no nontrivial minuscule representations in types $G_2$,
$F_4$ or $E_8$.  Except for these types, any representation can be
embedded into a tensor product of minuscule representations.

For example, the adjoint representation of $\GL(n)$ is not minuscule,
but can be regarded as the tensor product of the vector
representation and its dual.  Slightly less trivial, the vector
representation of the odd orthogonal group is not minuscule, but
appears as a direct summand in $S\otimes S$, where $S$ is the spin
representation of the spin group $\Spin(2n+1)$.

Given a dominant weight $\lambda$, we associate to the irreducible
representation $V(\lambda)$ its \Dfn{crystal graph} $B_\lambda$.
This is a certain connected edge-coloured digraph with
$\dim V(\lambda)$ vertices, each labelled with a weight of the
representation.  Each edge of the crystal graph is labelled with
one of the simple roots of the root system, such that the weight of
the target of the edge is obtained from the weight of its source by
subtracting the simple root.  There is a unique vertex without
in-coming edges, the \Dfn{highest weight vertex}, and this vertex has
weight $\lambda$.  There is also a unique vertex without out-going
edges, the \Dfn{lowest weight vertex}.  The sum of the formal
exponentials of the weights of the vertices is the character of the
representation.  In particular, isomorphism of crystal graphs
corresponds to isomorphism of representations.  The direct sum of
representations is then associated with the disjoint union of the
corresponding crystal graphs.

For a minuscule representation $V(\lambda)$ of dominant weight
$\lambda$, the vertices of the
associated % (then called \Dfn{minuscule})
crystal graph $B_\lambda$ can be identified with the weights of
$V(\lambda)$.  The edges are given by the Kashiwara lowering
operators, as follows.  Let $\{\alpha_i:i\in I\}$ be the set of
simple roots and $s_i\in W$ be the simple reflection corresponding to
$\alpha_i$.  Then there is a coloured edge
$\mu\stackrel{i}{\to}\mu-\alpha_i$ provided that
$s_i(\mu) = \mu-\alpha_i$.

There is a (relatively) simple way to construct the crystal graph of
a tensor product of representations given their individual crystal
graphs.  The vertices of the tensor product
$C_1\otimes\dots\otimes C_r$ of crystal graphs, corresponding to an
$r$-fold tensor product of representations, are the words of length
$r$ whose $i$-th letter is a vertex of $C_i$.  The weight of a vertex
in the tensor product is the sum of the weights of its letters.  In
this context, we refer to the highest weight vertices of the
connected components as \Dfn{highest weight words}.  Isolated
vertices correspond to copies of the trivial representation and
therefore have weight zero.  They are referred to as \Dfn{highest
  weight words of weight zero}.

There is an action of the so-called $r$-fruit cactus group $\fC_r$ on
the set of highest weight words in $r$-fold tensor products of
crystals $C_1,\dots,C_r$.  The cactus group $\fC_r$ has generators
$\cactus_{p,q}$, for $1\leq p < q\leq r$, satisfying certain
relations stated in Definition~\ref{dfn:cactus-groups-generators} of
Section~\ref{sec:crystals-local-rules}.  The generator
$\cactus_{p,q}$ acts by mapping highest weight words of
$C_1\otimes\dots\otimes C_r$ bijectively to highest weight words of
\[
C_1\otimes\dots\otimes C_{p-1}\otimes(C_q\otimes C_{q-1}\dots\otimes
C_p)\otimes C_{q+1}\otimes\dots\otimes C_r,
\]
see Definition~\ref{dfn:cactus-group-action}.  We define the
\Dfn{promotion} $\pr w$ of $w$ as $\cactus_{1,r}\cactus_{2,r} w$, and
the \Dfn{evacuation} $\ev w$ of $w$ as $\cactus_{1,r} w$.  These
generalise Schützenberger's maps of the same name, as we explain in
example~\ref{eg:exterior_powers} below.  We will see in
Lemma~\ref{lem:extend-s_1q} that the cactus group $\fC_r$ is already
generated by the elements $\cactus_{1,q}$ for $q\leq r$.  Therefore
it is essentially enough to understand evacuation.  An analogous
statement is true for promotion.

\begin{goal}
  Define a set of chord diagrams and a bijection between these and
  the highest weight words of weight zero which intertwines rotation
  of diagrams with promotion of words.  Additionally, determine the
  action of evacuation on the set of chord diagrams.
\end{goal}

In the following we briefly survey the cases in which (partial)
solutions to this problem are known.

Recall that dominant weights of $\SL(n)$, $\Sp(2n)$ and $\SO(2n+1)$
are vectors of length $n$ with weakly decreasing non-negative integer
entries.  Therefore, dominant weights can be identified with integer
partitions into at most $n$ parts in these cases.  A dominant weight
of $\Spin(2n+1)$ is a vector of length $n$ with weakly decreasing
non-negative half-integer entries, such that either all entries are
integers or none of them.  Finally, a dominant weight of $\GL(n)$ is
a vector of length $n$ with weakly decreasing integer entries, a
so-called staircase.  

Throughout, we denote the $i$-th unit vector by $e_i$.  To improve
readability, we will use vector notation for weights - often dropping
commas and parentheses, and specify letters of highest weight words
as linear combinations of unit vectors $e_i$.

To any highest weight word $w=w_1\dots w_r$ in a tensor product of
crystals $C_1\otimes\dots\otimes C_r$ we bijectively associate a
sequence of dominant weights going under names like semistandard,
oscillating, alternating, vacillating \Dfn{tableau}.  We call the
final weight $\mu=(\mu_1,\dots,\mu_n)$ of such a sequence, which will
also be the weight of the word $w$, the \Dfn{shape} of the tableau.
If $\mu$ is zero, we say that the tableau is of \Dfn{empty shape}.
We denote the zero weight by $\emptyset$.

Suppose now that in each crystal $C_i$, $1\leq i\leq r$, all vertices
have distinct weight.  For example, this is the case when all the
$C_i$ correspond to minuscule representations.  Then the tableau is
given by the sequence $\emptyset=\wgt^0,\wgt^1,\dots,\wgt^r=\wgt$,
where $\wgt^q=\sum_{i=1}^q \wt(w_i)$ is the sum of the weights of the
first $q$ letters.  In this case, one can recover the letters of the
highest weight word via the successive differences
$\wt(w_i)=\wgt^i-\wgt^{i-1}$.

In the examples below, the only exception to this rule is the case of
alternating tableaux for the adjoint representation of $\GL(n)$,
which we will explain separately.

\begin{example}[standard and semistandard tableaux]\label{eg:exterior_powers}
  Let $V$ be the vector representation of $\SL(n)$, and let each $C_i$,
  for $1\leq i\leq r$, be a copy of the corresponding crystal.  Then the
  highest weight words can be identified with standard Young tableaux
  of size $r$ with at most $n$ columns: the position of the unique
  entry equal to $1$ in $w_i$ is the column of the tableau in which
  the number $i$ appears.  Since the weight lattice of $\SL(n)$ is
  the image of $\ZZ^n$ in the quotient of $\RR^n$ by the span of
  $(1,\dots,1)$, a highest weight word has weight zero if and only if
  all $n$ columns of the corresponding tableau have the same length.

  More generally, for any sequence of positive integers
  $\alpha=\alpha_1,\dots\alpha_r$, let $C_i$ be the crystal
  corresponding to the $\alpha_i$-th exterior power of the vector
  representation of $\SL(n)$.  Then $w$ is a highest weight word if
  and only if $w_i$ has exactly $\alpha_i$ entries equal to $1$, all
  others $0$, and $\wgt^q$ is dominant for all $q\leq r$.  Therefore,
  a highest weight word can be identified with a semistandard Young
  tableau of type $\alpha$ having at most $n$ columns.\footnote{We
    are using slightly a nonstandard convention here.  More
    traditionally, one would use \emph{dual} semistandard tableaux
    instead, with entries in \emph{columns weakly increasing},
    entries in \emph{rows strictly increasing}, and at most $n$
    \emph{rows}.  The positions containing a $1$ in $w_i$ then
    designate the \emph{rows} of the tableau containing the number
    $i$.}

  One can show that in this case the generator $\cactus_{1,r}$ of the
  cactus group is precisely Schützenberger's evacuation of
  semistandard Young tableaux with largest entry at most $r$, and
  $\cactus_{1,r} \cactus_{2,r}$ is Schützenberger's promotion.  Using
  evacuation as a building block, the action of the cactus groups on
  semistandard Young tableaux was studied by Chmutov, Glick and
  Pylyavskyy~\cite{BK}.  As an aside, we remark that the generators
  $\cactus_{i,i+2}$ encode Assaf's dual equivalence graph.

  A diagrammatic basis for the invariant space was recently
  constructed by Cautis, Kamnitzer and Morrison~\cite{MR3263166},
  generalising Kuperberg's webs for $\SL(2)$ and $\SL(3)$,
  see~\cite{MR1403861}.  However, only Kuperberg's web bases are
  preserved by rotation.  For these, Petersen, Pylyavskyy and
  Rhoades~\cite{MR2519848} and Patrias~\cite{MR3861768} demonstrated
  that the growth algorithm of Khovanov and Kuperberg
  in~\cite{MR1680395} intertwines promotion with rotation.
\end{example}

\begin{example}[oscillating tableaux]\label{eg:oscillating}
  Let $V$ be the vector representation of $\Sp(2n)$ and let $C_i$ be
  the corresponding crystal, for $1\leq i\leq r$.  Then $w$ is a
  highest weight word if and only if $w_i$ is in
  $\{\pm e_j : 1\leq j \leq n\}$, and $\wgt^q$ is dominant for
  $q\leq r$.  The corresponding tableau is called an $n$-symplectic
  oscillating tableau.

  For example, the $1$-symplectic oscillating tableaux of length
  three are
  \[
  (\emptyset,1,2,3),\quad (\emptyset,1,2,1),\quad\text{and}\quad
  (\emptyset,1,\emptyset,1).
  \]
  The corresponding highest weight words are
  \[
  e_1\,e_1\,e_1,\quad e_1\,e_1\,\text{-}e_1,\quad\text{and}\quad e_1\,\text{-}e_1\,e_1.
  \]

  As a further example, the oscillating tableau
  \[
  \osc = (\emptyset,1,11,21,2,21,11,21,211,21)
  \]
  has length $9$ and shape $21$.  It is $3$-symplectic (since no
  partition has four parts) but it is not $2$-symplectic (since there
  is a partition with three parts).  The corresponding highest weight
  word is
  $w=e_1\,e_2\,e_1\,\text{-}e_2\,e_2\,\text{-}e_1\,e_1\,e_3\,\text{-}e_3$.

  A suitable set of chord diagrams is the set of $(n+1)$-noncrossing
  perfect matchings of $\{1,\dots,r\}$, see
  Section~\ref{sec:symplectic-results} for definitions.  A surjection
  from the set of perfect matchings to a basis of the invariant
  subspace of $\otimes^r V$ was given by Brauer~\cite{MR1503378}.
  Sundaram~\cite{MR2941115} provided a bijection between the set of
  $(n+1)$-noncrossing perfect matchings and $n$-symplectic
  oscillating tableaux of empty shape.
  Theorem~\ref{thm:rotation-promotion-matchings} below shows that
  this bijection intertwines rotation with promotion, and reversal
  with evacuation.

  A variation of this example is obtained by replacing $V$ with the
  $k$-th symmetric power of the vector representation.  A suitable
  set of chord diagrams indexing a basis of the space of invariant
  tensors was defined by Rubey and Westbury~\cite{RubeyWestbury2015,
    RubeyWestbury}.  Briefly, partition the set $\{1,\dots,kr\}$ into
  $r$ blocks of $k$ consecutive elements.  Then a chord diagram is an
  $(n+1)$-noncrossing perfect matching of this set, such that no pair
  is contained in a block and, if two pairs cross, the four elements
  are in four distinct blocks.
\end{example}

% \begin{example}[staircase tableaux]
%   Let $V$ be the vector representation of $\GL(n)$ and $V^\ast$ is
%   its dual.  Let $\epsilon=(\epsilon_1,\dots,\epsilon_r)$ in
%   $\{\pm 1\}^r$ and let $C_i$ be the crystal corresponding to $V$ if
%   $\epsilon_i = 1$ and the crystal corresponding to $V^\ast$ if
%   $\epsilon_i=-1$.  Then $w$ is a highest weight word if and only if
%   $w_i$ is in $\{\epsilon_i e_j : 1\leq j \leq n\}$, and $\wgt^q$ is
%   dominant for $q\leq r$.  The corresponding tableaux are called
%   staircase tableaux.
% \end{example}

\begin{example}[alternating tableaux]\label{eg:alternating}
  Let $\mathfrak{gl}_n$ be the adjoint representation of $\GL(n)$.
  This representation is not minuscule, but we can identify it with
  $V\otimes V^\ast$ where $V$ is the vector representation of
  $\GL(n)$ and $V^\ast$ is its dual.  Let $C_i$, for $1\leq i\leq r$,
  be the crystal corresponding to $V\otimes V^\ast$.  Thus, the
  letters of a highest weight word $w$ are pairs
  $(e_k,\text{-}e_\ell)$ of weight $e_k-e_\ell$, with
  $1\leq k,\ell\leq n$.  The corresponding $\GL(n)$-alternating
  tableau is the sequence of dominant weights
  $\emptyset\!=\!\wgt^0,\wgt^1,\ldots,\wgt^{2r}\!=\!\wgt$, where
  $\wgt^{2q} = \sum_{i=1}^q \wt(w_i)$, and
  $\wgt^{2q+1} = \wgt^{2q} + e_k$ when
  $w_{q+1}=(e_k,\text{-}e_\ell)$.

  A word $w$ is of highest weight if and only if $\wgt^p$ is dominant
  for $p\leq 2r$.  Given a $\GL(n)$-alternating tableau, one can
  recover the letter $w_i$ of the corresponding highest weight word
  as $w_i=(e_k,\text{-}e_\ell)$ with $e_k=\wgt^{2i-1}-\wgt^{2i-2}$
  and $\text{-}e_\ell=\wgt^{2i}-\wgt^{2i-1}$.

  For example, the $\GL(2)$-alternating tableaux of length two are
  \begin{align*}
  &(00,10,00,10,00),\quad
   (00,10,00,10,1\bar1),\quad
   (00,10,1\bar1,10,00),\\
  &(00,10,1\bar1,10,1\bar1),\quad
   (00,10,1\bar1,2\bar1,1\bar1),\quad
   (00,10,1\bar1,2\bar1,2\bar2),
  \end{align*}
  writing $\bar 1$ in place of $-1$, etc., for better readability.
  The corresponding highest weight words are
  \begin{align*}
    &(e_1,\text{-}e_1)\,(e_1,\text{-}e_1),\quad (e_1,\text{-}e_1)\,(e_1,\text{-}e_2),\quad (e_1,\text{-}e_2)\,(e_2,\text{-}e_1),\\
    &(e_1,\text{-}e_2)\,(e_2,\text{-}e_2),\quad (e_1,\text{-}e_2)\,(e_1,\text{-}e_1),\quad (e_1,\text{-}e_2)\,(e_1,\text{-}e_2).
  \end{align*}

  For $n$ large enough, a suitable set of chord diagrams is the set
  of permutations of $\{1,\dots,r\}$, see
  Section~\ref{sec:adjoint-results} for definitions.  We provide a
  bijection between this set and $\GL(n)$-alternating tableaux of
  empty shape, see Theorem~\ref{thm:rotation-promotion-permutations}
  below.  This bijection intertwines rotation with promotion, and
  reverse-complement with evacuation.
\end{example}

\begin{example}[fans of Dyck paths]
  Let $S$ be the spin representation of the spin group $\Spin(2n+1)$
  and let $\lambda_i = \frac{1}{2}\sum_{j=1}^n e_j$ be its dominant
  weight.  Then $w$ is a highest weight word if and only if
  $w_i = (\pm\frac{1}{2},\dots,\pm\frac{1}{2})$ and
  $\wgt^q$ is dominant for all $q\leq r$.

  Therefore, a highest weight word $w$ of weight zero can be
  identified with a fan of $n$~Dyck paths of length $r$: the first
  entry of $w_i$ is $\frac{1}{2}$ if and only if the top most Dyck
  path has an up-step at position $i$.  In general, the $j$-th entry
  of $w_i$ is $\frac{1}{2}$ if and only if the $j$-th Dyck path has
  an up-step at position $i$.  One can show that $\ev$ acts on these
  as reversal.

  In general, no suitable set of chord diagrams is known.  For $n=2$
  there is an exceptional isomorphism between $S$ and the vector
  representation of $\Sp(4)$.  Thus, the results
  for oscillating tableaux apply in this case.
\end{example}

\begin{example}[vacillating tableaux]

  Let $V$ be the vector representation of $\SO(2n+1)$ and let
  $\lambda_i=e_1$ be its dominant weight.  Then $w$ is a highest
  weight word if and only if $w_i$ is in
  $\{\pm e_j : 1\leq j \leq n\}\cup\{0\}$, $\mu^q$ is dominant for
  $q\leq r$ and $w_i\neq 0$ if $\mu^{i-1}$ contains an entry equal to
  $0$.  The corresponding tableaux are called vacillating tableaux
  and can be identified with $n$-fans of Riordan paths,
  see~\cite{MR3875015}.

  In general, no suitable set of chord diagrams is known.  For $n=1$
  there is an exceptional isomorphism between $V$ and the adjoint
  representation of $\SL(2)$.  Thus we obtain a
  bijection between noncrossing set partitions without singletons and
  highest weight words of weight zero with the desired properties
  from the results in Section~\ref{sec:adjoint-results} below.  For
  $n=2$, a bijection between a basis of the invariant subspace of
  $\otimes^r V$ and certain chord diagrams follows from Kuperberg's
  webs~\cite{MR1403861}.
\end{example}

\section{Results}
\label{sec:results}

In this section we present combinatorial realisations of promotion
and evacuation on highest weight words for the vector representation
of the symplectic group and the adjoint representation of the general
linear group.  As it turns out, the former will essentially follow
from the latter.  The proofs are delegated to
Section~\ref{sec:proofs}.

\subsection{The vector representation of the symplectic groups}
\label{sec:symplectic-results}

Let us first recall Sundaram's definition of $n$-symplectic
oscillating tableaux from example~\ref{eg:oscillating}.  As explained
there, these are in bijection with the highest weight words in a
tensor power of the crystal corresponding to the vector
representation of the symplectic group $\Sp(2n)$.

\begin{defn}[Sundaram~\cite{MR2941115}]\label{defn:oscillating}
  An \Dfn{$n$-symplectic oscillating tableau}~$\osc$ of length~$r$
  and (final) shape $\wgt$ is a sequence of partitions
  \[
  \osc=(\emptyset\!=\!\wgt^0,\wgt^1,\ldots,\wgt^r\!=\!\wgt)
  \]
  such that the Ferrers diagrams of two consecutive partitions differ
  by exactly one cell, and each partition $\wgt^i$ has at most $n$
  non-zero parts.
\end{defn}

Recall that a \Dfn{partial standard Young tableau} is a filling of
the Ferrers diagram with distinct non-negative integers such that
entries in rows and columns are strictly increasing.

A now classic bijection due to Sundaram~\cite{MR2941115} maps an
oscillating tableau $\osc$ of length $r$ and shape $\mu$ to a pair
$\big(\M(\osc), \M_T(\osc)\big)$, consisting of a perfect matching of
a subset of $\{1,\dots,r\}$ and a partial standard Young tableau of
shape $\mu$, whose entries form the complementary subset.  We present
Roby's~\cite{MR2716353} description of this bijection in
Section~\ref{sec:growth-diagrams}.

The \Dfn{reversal} $\rev m$ of a perfect matching $m$ of a subset of
$\{1,\dots,r\}$ is obtained by replacing each pair~$(i,j)$
with~$(r+1-j, r+1-i)$.  Our first main result relates this operation
to evacuation as follows.
\begin{thm}\label{thm:s1q-partial-matchings}
  Let $\osc$ be an $n$-symplectic oscillating tableau.  Then
  $\M(\ev\osc)$ is the reversal of $\M(\osc)$ and $\M_T(\ev\osc)$ is
  the Schützenberger evacuation of $\M_T(\osc)$.
\end{thm}

If the oscillating tableau $\osc$ is of empty shape the tableau
$\M_T(\osc)$ is empty and $\M(\osc)$ is a perfect matching of
$\{1,\dots,r\}$.  We use chord diagrams to visualise perfect
matchings of $\{1,\dots,r\}$, drawing their pairs as (straight)
diagonals connecting the vertices of a counterclockwise labelled
regular $r$-gon, see Figure~\ref{fig:chord-diagrams}.

A perfect matching is \Dfn{$(n+1)$-noncrossing} if it contains at
most~$n$ pairs that mutually cross in its chord diagram.  It follows
from Sundaram's bijection that these are precisely the perfect
matchings corresponding to $n$-symplectic oscillating tableaux of
empty shape.

Using the visualisation as a chord diagram, the \Dfn{reversal} of a
perfect matching is obtained by reflecting the diagonals of the
diagram on a well-chosen axis.  The \Dfn{rotation} $\rot m$ of a
matching $m$ is obtained by replacing each pair $(i,j)$ with the pair
$\big(i\pmod r +1, j\pmod r + 1\big)$.  Visually, this corresponds to
a rotation of the diagonals of the diagram.

\begin{thm}\label{thm:rotation-promotion-matchings}
  The bijection $\M$ between $n$-symplectic oscillating tableaux of
  empty shape and $(n+1)$-noncrossing perfect matchings intertwines
  promotion and rotation, and evacuation and reversal:
  \[
  \rot\M(\osc) = \M(\pr\osc)\text{ and }\rev\M(\osc) = \M(\ev\osc).
  \]
\end{thm}

\begin{rmk}
  Consider the natural embedding $\iota$ of the set of $n$-symplectic
  oscillating tableaux into the set of $(n+1)$-symplectic oscillating
  tableaux.  Since rotation and reversal preserve the maximal
  cardinality of a crossing set of diagonals in the chord diagram, we
  obtain the remarkable fact that $\pr\iota(\osc) = \iota(\pr\osc)$
  and $\ev\iota(\osc) = \iota(\ev\osc)$.
\end{rmk}

\subsection{The adjoint representation of the general linear groups}
\label{sec:adjoint-results}

We recall from example~\ref{eg:alternating} Stembridge's definition
of $\GL(n)$-alternating tableaux, which are in bijection with the
highest weight words in a tensor power of the crystal corresponding
to the adjoint representation of $\GL(n)$.

\begin{defn}[Stembridge~\cite{MR899903}]\label{defn:alternating}
  A \Dfn{staircase} is a dominant weight of $\GL(n)$, that is, a
  vector in $\ZZ^n$ with weakly decreasing entries.  A
  \Dfn{$\GL(n)$-alternating tableau} $\alt$ of length $r$ and shape
  $\wgt$ is a sequence of staircases
  \[
  \alt=(\emptyset\!=\!\wgt^0,\wgt^1,\ldots,\wgt^{2r}\!=\!\wgt)
  \]
  such that
  \begin{itemize}\setlength{\itemsep}{0pt}
  \item[] for even $i$, $\wgt^{i+1}$ is obtained from $\wgt^i$ by
    adding $1$ to an entry, and
  \item[] for odd $i$, $\wgt^{i+1}$ is obtained from $\wgt^i$ by
    subtracting $1$ from an entry.
  \end{itemize}
\end{defn}

In Section~\ref{sec:growth-diagrams}, we introduce a bijection
similar in spirit to Sundaram's.  It maps an alternating tableau
$\alt$ of length $r$ and shape $\mu$ to a triple
$\big(\Perm(\alt), \Perm_P(\alt), \Perm_Q(\alt)\big)$, consisting of
a bijection $\Perm(\alt):R\to S$ between two subsets of
$\{1,\dots,r\}$, and two partial standard Young tableaux
$\Perm_P(\alt)$ and $\Perm_Q(\alt)$.  The shapes of these tableaux
are obtained by separating the positive and negative entries of
$\mu$.  The entries of the first tableau then form the complementary
subset of $R$, the entries of the second form the complementary
subset of $S$.

For a bijection $\pi: R\to S$ between two subsets of $\{1,\dots,r\}$,
the \Dfn{reverse-complement} $\reco\pi$ maps $r+1-i$ to $r+1-\pi(i)$.

\begin{thm}\label{thm:s1q-partial-permutations}
  Let $\alt$ be a $\GL(n)$-alternating tableau of length
  $r\leq \lfloor\frac{n+1}{2}\rfloor$.  Then $\Perm(\ev\alt)$ is the
  reverse-complement of $\Perm(\alt)$, and $\Perm_P(\ev\alt)$ and
  $\Perm_Q(\ev\alt)$ are obtained by applying Schützenberger's
  evacuation to $\Perm_P(\alt)$ and $\Perm_Q(\alt)$ respectively:
  \[
  \big(\Perm_P(\ev\alt), \Perm_Q(\ev\alt)\big) =
  \big(\ev\Perm_P(\alt), \ev\Perm_Q(\alt)\big).
  \]
\end{thm}

Similarly to Sundaram's map between oscillating tableaux and
matchings, if the alternating tableau $\alt$ is of empty shape the
tableaux $\Perm_P(\alt)$ and $\Perm_Q(\alt)$ are empty and
$\Perm(\alt)$ is a permutation.  We use chord diagrams to visualise a
permutation $\pi$, drawing an arc from vertex $i$ to vertex $\pi(i)$
in a counterclockwise labelled regular $r$-gon, see
Figure~\ref{fig:chord-diagrams}.

Using this visualisation, the \Dfn{reverse-complement} of a
permutation is obtained by reflecting the diagonals of the diagram on
a well-chosen axis.  The \Dfn{rotation} $\rot\pi$ of a permutation
$\pi$ is obtained by replacing each arc $\big(i, \pi(i)\big)$ with
the arc \mbox{$\big(i\pmod r +1, \pi(i)\pmod r + 1\big)$}.  Visually,
this corresponds to a rotation of the diagonals of the
diagram.

\begin{thm}\label{thm:rotation-promotion-permutations}
  For $n\geq r-1$ and also for $n\leq 2$ the bijection $\Perm$
  between $\GL(n)$-alternating tableaux of empty shape of length $r$
  and permutations intertwines promotion and rotation:
  \[
  \rot\Perm(\alt) = \Perm(\pr\alt).
  \]
  For even $n\geq r$ and for odd $n\geq r-1$, it intertwines
  evacuation and reverse-complement:
  \[
  \reco\Perm(\alt) = \Perm(\ev\alt).
  \]
  For $n\leq2$ and arbitrary $r$, it intertwines evacuation and
  inverse-reverse-complement:
  \[
  \reco\Perm(\alt)^{-1} = \Perm(\ev\alt).
  \]
\end{thm}
\begin{rmk}
  Note that the case $n\leq 2$ is special.  As we will show in
  Section~\ref{sec:GL2}, our bijection identifies
  $\GL(2)$-alternating tableaux of empty shape in a natural way with
  noncrossing partitions, which form an invariant set under rotation.
  In fact, this set coincides with the web basis for $\GL(2)$.
  Moreover, in this case the evacuation of an alternating tableau is
  simply its reversal. In terms of noncrossing set partitions, the
  inverse-reverse-complement of the corresponding permutation is the
  mirror image of the set partition.
\end{rmk}

\begin{rmk}
  It is tempting to regard a $\GL(n)$-alternating tableau as a
  sequence of pairs of partitions by separating the positive and
  negative entries. Indeed, this is what we will do in
  Sections~\ref{sec:growth-diagrams} and~\ref{sec:proofs} to define
  our bijection and prove Theorems~\ref{thm:s1q-partial-permutations}
  and~\ref{thm:rotation-promotion-permutations}.

  One might then think that promotion can be defined directly in
  terms of these sequences without reference to $n$.  However, this
  is not the case.  For $n>2$, promotion does not preserve the
  maximal number of non-zero entries in a vector of an alternating
  tableau.

  In fact, it is not clear whether there is an embedding $\iota$ of
  the set of $\GL(n)$-alternating tableaux in the set of
  $\GL(n+1)$-alternating tableaux such that
  $\pr\iota(\alt) = \iota(\pr\alt)$.  In spite of this, we prove a
  certain stability phenomenon for promotion of alternating tableaux
  in Theorem~\ref{thm:stability}.
\end{rmk}

\section{The cactus groups, local rules, promotion and evacuation}
\label{sec:crystals-local-rules}
In this section, following Henriques and Kamnitzer~\cite{MR2219257},
we define promotion and evacuation of highest weight words as an
action of certain elements of the $r$-fruit cactus group on $r$-fold
tensor products of crystals.  Then, following van
Leeuwen~\cite{MR1661263} and Lenart~\cite{MR2361854} we encode the
action of the cactus group by certain local rules, generalising
Fomin's.

\subsection{The cactus group and its action}
Let us first define the cactus groups.
\begin{defn}\label{dfn:cactus-groups-generators}
  The \Dfn{$r$-fruit cactus group}, $\fC_r$, has generators $\cactus_{p,q}$
  for $1\le p<q\le r$ and defining relations
  \begin{itemize}\setlength{\itemsep}{0pt}
  \item $\cactus_{p,q}^2=1$
  \item $\cactus_{p,q}\, \cactus_{k,l}=\cactus_{k,l}\, \cactus_{p,q}$ if $q<k$ or $l<p$
  \item $\cactus_{p,q}\, \cactus_{k,l}=\cactus_{p+q-l,p+q-k}\, \cactus_{p,q}$ if $p\le k<l\le q$
  \end{itemize}
  For convenience we additionally define $\cactus_{p,p}=1$.
\end{defn}

It may be useful to think of the generators as being indexed by
intervals in $\{1,\dots,r\}$.  Then the second relation can be
rephrased by saying that two generators commute if they are indexed
by disjoint intervals.  The third relation is applicable when one
interval is contained in the other, in which case the inner interval
is reflected within the outer interval.

The following lemma shows that it is sufficient to define the action
of the composites $\cactus_{1,q} \cactus_{2,q}$ for $2\leq q\leq r$.
The first relation was observed by White~\cite[Lem.~2.3]{MR3845719},
the second is in analogy to Schützenberger's original definition of
evacuation of standard Young tableaux in~\cite[Sec.~5]{MR0190017}.
\begin{lemma}\label{lem:extend-s_1q}
  We have
  \[
    \cactus_{p,q} = \cactus_{1,q}\cactus_{1,q-p+1}\cactus_{1,q}%
    \quad\text{and}\quad%
    \cactus_{1,q} = \cactus_{1,2} \cactus_{2,2}\; \cactus_{1,3} \cactus_{2,3}\; \dots\; \cactus_{1,q} \cactus_{2,q}.
  \]
\end{lemma}
\begin{proof}
  The first equality is obtained from the third defining relation by
  replacing $p,q,k,\ell$ with $1,q,p$ and $q$ respectively.  The
  second equality follows from
  $\cactus_{1,\ell} \cactus_{2,\ell} = \cactus_{1,\ell-1}
  \cactus_{1,\ell}$,
  which is also an instance of the third defining relation.
\end{proof}

Henriques and Kamnitzer~\cite{MR2219257} defined an action of the
cactus group on $r$-fold tensor products of crystals in terms of the
commutor, which in turn is defined using Lusztig's involution.  Let
us first briefly recall the latter, as introduced
in~\cite{MR1182165}.
\begin{defn}
  Let $B$ be a crystal graph associated with an irreducible
  representation.  \Dfn{Lusztig's involution} $\lusztig$ maps the
  unique highest weight vertex of $B$ to its unique lowest weight
  vertex, and the Kashiwara lowering operator $f_i$ to the Kashiwara
  raising operator $e_{i^\ast}$, where $i\mapsto i^\ast$ is the
  Dynkin diagram automorphism specified by
  $\alpha_{i^\ast} = -w_0(\alpha_i)$, and $w_0$ is the longest
  element of the Weyl group.  This definition is extended to
  arbitrary crystals by applying the involution to each connected
  component separately.
\end{defn}

Lusztig's involution is not a morphism of crystals, which would have
to map highest weight vertices to highest weight vertices.  For the
Cartan type $A_n$ crystal $B_\lambda$ of semistandard Young tableaux
of shape $\lambda$, Lusztig's involution is precisely
Schützenberger's evacuation of semistandard Young tableaux with
largest entry at most $n+1$.

\begin{defn}
  For two crystals $A$ and $B$, the \Dfn{commutor} is the crystal morphism
\begin{align*}
&\sigma_{A,B}: A\otimes B\to B\otimes A\\
&a\,b\mapsto \lusztig\big(\lusztig(b)\,\lusztig (a)\big).
\end{align*}
\end{defn}
We can now define the action of the cactus group.
\begin{defn}\label{dfn:cactus-group-action}
  The action of $\fC_r$ on words in $C_1\otimes\dots\otimes C_r$ is
  defined inductively by letting $\cactus_{p,p+1}$ act as
  $1\otimes\sigma_{C_p,C_{p+1}}\otimes 1$ and $\cactus_{p,q}$ as
  $(1\otimes\sigma_{C_p,C_{p+1}\otimes\dots\otimes C_q}\otimes 1
  )\circ\cactus_{p+1, q}$
  for $q > p+1$, where $1$ denotes the identity map on a crystal.
\end{defn}

The action can be expressed more explicitly directly in terms of
Lusztig's involution.
\begin{prop}
  Let $w=w_1\dots w_r$ be a word in $C_1\otimes\dots\otimes C_r$,
  then
    \[
    \cactus_{p,q} w_1\dots w_r = w_1\dots w_{p-1}\lusztig\big(\lusztig (w_q)\lusztig(w_{q-1})\dots\lusztig(w_p)\big) w_{q+1}\dots w_r.
    \]
\end{prop}
\begin{proof}
  This follows by induction on $q-p$ and the fact that $\lusztig$ is
  an involution.
\end{proof}

\begin{defn}
  The \Dfn{promotion} $\pr w$ of $w$ is
  $\cactus_{1,r}\cactus_{2,r} w$, and the \Dfn{evacuation} $\ev w$ of
  $w$ is $\cactus_{1,r} w$.  For a tableau $\mathcal T$ corresponding
  to a highest weight word $w$, we use $\pr\mathcal T$ and
  $\ev\mathcal T$ to denote the tableaux corresponding to $\pr w$ and
  $\ev w$.
\end{defn}

It will be convenient to express promotion as a commutor, as follows.
\begin{prop}
  $\cactus_{1,q}\cactus_{2,q}(w)=\sigma_{C_1,C_2\otimes\dots\otimes C_q}(w)$.
\end{prop}
\begin{proof}
  This is immediate from the definition of the action of
  $\cactus_{1,q}$ and from the fact that $\cactus_{2,q}^2 = 1$.
\end{proof}

\subsection{Promotion and evacuation via local rules}

We now follow Lenart's approach~\cite{MR2361854} and realise the
action of the cactus group using van Leeuwen's local
rules~\cite[Rule~4.1.1]{MR1661263}, which generalise
Fomin's~\cite[A~1.2.7]{MR1676282}.

From now on we restrict ourselves to minuscule representations.
However, let us remark that van Leeuwen also provides a local rule
that applies to quasi-minuscule representations, which makes our
approach viable for arbitrary representations.

For minuscule representations, van Leeuwen's rules involve obtaining
the unique dominant representative of a weight.
\begin{defn}\label{def:dominant-weight}
  Let $\lambda$ be a weight of a representation of a Lie group with
  Weyl group $W$.  Then \Dfn{$\dom_W(\lambda)$} is the unique
  dominant representative of the $W$-orbit $W\lambda$.
\end{defn}
\begin{example}
  The Weyl group of $\SL(n)$ is the symmetric group $\fS_n$.  Thus,
  $\dom_{\fS_n}$ returns its argument sorted into decreasing order.
\end{example}
\begin{example}
  The Weyl group of $\Sp(2n)$ is the hyperoctahedral group $\fH_n$ of
  signed permutations of $\{\pm 1,\dots,\pm n\}$.  Thus, the dominant
  representative $\dom_{\fH_n}(\lambda)$ of a weight $\lambda$ is
  obtained by sorting the absolute values of its entries into
  decreasing order.
\end{example}

We can now define the local rule.
\begin{defn}\label{def:local-rules}
  Let $A$ be a crystal and $B$ and $C$ be crystals of minuscule
  representations.  Then the \Dfn{local rule}
  \[
  \tau^A_{B,C}\colon A\otimes B\otimes C \rightarrow A\otimes
  C\otimes B
  \]
  is an isomorphism of crystals defined for highest weight words
  $a\,b\,c$ as follows: let $\kappa$ be the weight of $a$, let
  $\lambda$ be the weight of $a\,b$ and let $\nu$ be the weight of
  $a\,b\,c$.  Then
  \[
  \tau^A_{B,C}(a\,b\,c)= a\,\hat c\,\hat b,
  \]
  where %, regarding $\kappa$, $\lambda$, $\mu$ and $\nu$ as vectors,
  \[
  \hat c = \mu-\kappa\quad\text{and}\quad \hat
  b=\nu-\mu\quad\text{with}\quad\mu = \dom_W(\kappa+\nu-\lambda).
  \]
  We represent this by the following diagram:
  \begin{equation}\label{eq:local}
    \begin{tikzpicture}[scale=1,baseline=(current bounding box.center)]
      \node (ka) at (0,0) {$\kappa$};
      \node (la) at (0,1) {$\lambda$};
      \node (nu) at (1,1) {$\nu$};
      \node (mu) at (1,0) {$\mu$};
      \draw[->] (ka) -- (la) node [midway, left] {$b$};
      \draw[->] (mu) -- (nu) node [midway, right] {$\hat b$};
      \draw[->] (la) -- (nu) node [midway, above] {$c$};
      \draw[->] (ka) -- (mu) node [midway, below] {$\hat c$};
    \end{tikzpicture}
  \end{equation}
  Since any isomorphism between crystals is determined by specifying
  a bijection between the corresponding highest weight words, this
  definition can be extended to an isomorphism between
  $A\otimes B\otimes C$ and $A\otimes C\otimes B$ by applying the
  lowering operators.
\end{defn}
\begin{rmk}
  When $B=C$ is the crystal associated with the vector representation
  of $\SL(n)$, the local rule~\eqref{eq:local} is Fomin's for
  Schützenberger's jeu de taquin~\cite[A~1.2.7]{MR1676282}.  More
  explicitly, if $\lambda$ is the only partition of its size that
  contains $\kappa$ and is contained in $\nu$, then $\mu=\lambda$.
  Otherwise there is a unique such partition different from
  $\lambda$, and this is $\mu$.
\end{rmk}

\begin{rmk}\label{rmk:symmetry-of-local-rule}
  As in the classical case the local rule is symmetric in the sense
  that $\mu = \dom_W(\kappa+\nu-\lambda)$ if and only if
  $\lambda = \dom_W(\kappa+\nu-\mu)$, see
  \cite[Lem.~4.1.2]{MR1661263}.
\end{rmk}

\begin{example}
  Let $C$ be the crystal associated with the vector representation of
  $\SL(3)$, and let $A$ and $B$ be the crystal associated with its
  exterior square.  Let $a=e_1+e_2$, let $b=e_1+e_3$ and let $c=e_2$.
  Then $a\,b\,c$ is a highest weight word in $A\otimes B\otimes C$.
  We have $\kappa=(1,1,0)$, $\lambda=(2,1,1)$ and $\nu=(2,2,1)$.

  Thus, since $\dom_{\fS_n}$ sorts its argument into decreasing
  order,
  \[
  \mu = \dom_{\fS_n}(\kappa + \nu - \lambda) = \dom_{\fS_n}(1,2,0) =
  (2,1,0),%
  \quad\text{and}\quad \hat b=e_2+e_3,\quad \hat c=e_1.
  \]
\end{example}
\begin{example}
  Let $B$ and $C$ be the crystal associated with the vector
  representation of $\Sp(4)$ and let $A=\otimes^2 B$.  Let
  $a=e_1\,e_1$, let $b=e_2$ and let $c=-e_2$.  Then $a\,b\,c$ is a
  highest weight word in $A\otimes B\otimes C$.  We have
  $\kappa=(2,0)$, $\lambda=(2,1)$ and $\nu=(2,0)$.

  Thus, since $\dom_{\fH_n}$ takes absolute values and then sorts
  into decreasing order,
  \[
  \mu = \dom_{\fH_n}(\kappa + \nu - \lambda) = \dom_{\fH_n}(2,-1) =
  (2,1),%
  \quad\text{and}\quad \hat b=b,\quad \hat c=c.
  \]
\end{example}

\begin{thm}[\protect{\cite[Thm.~4.4]{MR2361854}}]\label{thm:local_commutor}
  Let $A$ and $B$ be crystals embedded into tensor products
  $A_1\otimes\dots\otimes A_k$ and $B_1\otimes\dots\otimes B_\ell$ of
  crystals of minuscule representations.  Let
  $w=w_1,\dots,w_{k+\ell}$ be a highest weight word in $A\otimes B$
  with corresponding tableau
  $\emptyset=\wgt^0,\wgt^1,\dots,\wgt^r=\wgt$.  Then
  $\sigma_{A,B}(w)$ can be computed as follows.  Create a
  $k\times \ell$ grid of squares as in \eqref{eq:local}, labelling
  the edges along the left border with $w_1,\dots,w_k$ and the edges
  along the top border with $w_{k+1},\dots,w_{k+\ell}$:
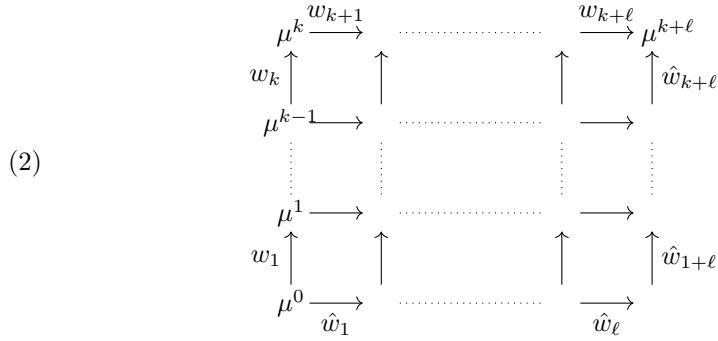
  \begin{equation}\label{eq:local_commutor}
    \tikzset{ short/.style={ shorten >=7pt, shorten <=7pt } }
    \begin{tikzpicture}[baseline=(current bounding box.center), scale=1.2]
      \node at (0,0) {$\wgt^0$};
      \node at (0,1) {$\wgt^1$};
      \node at (0,2) {$\wgt^{k-1}$};
      \node at (0,3) {$\wgt^k$};
      \node[right] at (4,3) {{\hbox to 5pt{\hss $\wgt^{k+\ell}$\hss}}};
      % vertical
      \draw[->,short] (0,0) -- (0,1) node [midway, left] {$w_1$};
      \draw[dotted,short] (0,1) -- (0,2) node {};
      \draw[->,short] (0,2) -- (0,3) node [midway, left] {$w_k$};
      \draw[->,short] (1,0) -- (1,1) node {};
      \draw[dotted,short] (1,1) -- (1,2) node {};
      \draw[->,short] (1,2) -- (1,3) node {};
      \draw[->,short] (3,0) -- (3,1) node {};
      \draw[dotted,short] (3,1) -- (3,2) node {};
      \draw[->,short] (3,2) -- (3,3) node {};
      \draw[->,short] (4,0) -- (4,1) node [midway, right] {$\hat w_{1+\ell}$};
      \draw[dotted,short] (4,1) -- (4,2) node {};
      \draw[->,short] (4,2) -- (4,3) node [midway, right] {$\hat w_{k+\ell}$};
      % horizontal
      \draw[->,short] (0,3) -- (1,3) node [midway, above] {$w_{k+1}$};
      \draw[dotted,short] (1,3) -- (3,3) node {};
      \draw[->,short] (3,3) -- (4,3) node [midway, above] {$w_{k+\ell}$};
      \draw[->,short] (0,2) -- (1,2) node {};
      \draw[dotted,short] (1,2) -- (3,2) node {};
      \draw[->,short] (3,2) -- (4,2) node {};
      \draw[->,short] (0,1) -- (1,1) node {};
      \draw[dotted,short] (1,1) -- (3,1) node {};
      \draw[->,short] (3,1) -- (4,1) node {};
      \draw[->,short] (0,0) -- (1,0) node [midway, below] {$\hat w_1$};
      \draw[dotted,short] (1,0) -- (3,0) node {};
      \draw[->,short] (3,0) -- (4,0) node [midway, below] {$\hat w_\ell$};
    \end{tikzpicture}
  \end{equation}
  For each square whose left and top edges are already labelled use
  the local rule to compute the labels on the square's bottom and
  right edges.  The labels $\hat w_1\dots \hat w_{k+\ell}$ of the
  edges along the bottom and the right border of the grid then form
  $\sigma_{A,B}(w)$.
\end{thm}
% By Lemma~\ref{lem:extend-s_1q} this determines $\cactus_{p,q} w$ for arbitrary $p$ and $q$.

\begin{example}
  Let $C_1 = C_2 = C_4$ be the crystal associated with the exterior
  square of $\SL(3)$ and let $C_3$ the crystal corresponding to its
  vector representation.  Then
  $w = e_1\text{+}e_2 \; e_1\text{+}e_3\; e_2\; e_1\text{+}e_3$ is
  the highest weight word corresponding to the semistandard tableau
  \[
    \young(112,234,4).
  \]
  The promotion $\pr w = \sigma_{C_1,C_2\otimes C_3\otimes C_4}(w)$
  can now be computed using Theorem~\ref{thm:local_commutor}.  We put
  the zero weight in the bottom-left corner.  Then, beginning with
  $\wgt^1=w_1$, we write down the sequence of cumulative weights
  $\wgt^1,\dots,\wgt^4$, with $\wgt^q = \sum_{i=1}^q w_i$ in the top
  row.  Finally, we successively apply the local
  rule~\eqref{eq:local} and fill in the weights in the bottom row.
  \begin{equation*}
    \tikzset{ short/.style={ shorten >=10pt, shorten <=10pt } }
    \begin{tikzpicture}[scale=1.5]
      \node at (0,0) {$000$};
      \node at (0,1) {$110$};
      \node at (1,0) {$110$};
      \node at (1,1) {$211$};
      \node at (2,0) {$210$};
      \node at (2,1) {$221$};
      \node at (3,0) {$311$};
      \node at (3,1) {$322$};
      % vertical
      \draw[->,short] (0,0) -- (0,1) node [midway, left] {$e_1\text{+}e_2$};
      \draw[->,short] (1,0) -- (1,1) node {};
      \draw[->,short] (2,0) -- (2,1) node {};
      \draw[->,short] (3,0) -- (3,1) node {};
      % horizontal
      \draw[->,short] (0,1) -- (1,1) node [midway, above] {$e_1\text{+}e_3$};
      \draw[->,short] (1,1) -- (2,1) node [midway, above] {$e_2$};
      \draw[->,short] (2,1) -- (3,1) node [midway, above] {$e_1\text{+}e_3$};
      \draw[->,short] (0,0) -- (1,0) node {};
      \draw[->,short] (1,0) -- (2,0) node {};
      \draw[->,short] (2,0) -- (3,0) node {};
    \end{tikzpicture}
  \end{equation*}
  From now on we omit edges and their labels, because the labels are
  determined by the weights.

  By Theorem~\ref{thm:local_commutor} the promotion of $w$ is given
  by the sequence of weights in the bottom row, together with the
  weight in the top-right corner.  The corresponding tableau is
  \[
    \young(113,244,3).
  \]

  As mentioned in the introduction, this definition of promotion
  coincides with the classical definition of promotion in terms of
  Bender-Knuth moves on tableaux when the crystals correspond to
  exterior powers of the vector representation of $\SL(n)$.
\end{example}

As an aside we obtain a formulation of the commutor, and therefore
also of promotion of highest weight words in crystals of minuscule
representations, analogous to the definition in terms of slides in
tableaux, as follows.
\begin{cor}\label{cor:sliding}
  Let $A$ and $B$ be crystals embedded into tensor products of
  crystals of minuscule representations. Let $a\in A$ and $b\in B$
  such that $a\,b$ is a highest weight word in $A\otimes B$, and let
  $\hat b\,\hat a= \sigma_{A,B}(a\,b)$ with $\hat b\in B$ and
  $\hat a\in A$.

  Then
  \begin{itemize}
  \item[] $\hat b$ is the highest weight vertex in the same component
    of $B$ as $b$, and
  \item[] $\hat a$ is a vertex of $A$ such that the weight of
    $\hat b\,\hat a$ equals the weight of $a\,b$.
  \end{itemize}

  In particular, if $A$ is a crystal of a minuscule representation,
  $\hat a$ is determined uniquely by its weight.
\end{cor}
\begin{proof}
  Let $B_\lambda$ be the component of $B$ containing $b$.  Because of
  the naturality of the commutor in $B$ (see property $(C1)$ in
  \cite{MR2361854}), $\sigma_{A,B}(a\,b)$ equals
  $\sigma_{A,B_\lambda}(a\,b)$.

  Since $a\,b$ is a highest weight word and the commutor is an
  isomorphism of crystals,
  $\sigma_{A,B_\lambda}(a\,b)=\hat b\,\hat a$ is also a highest
  weight word.  It follows that $\hat b$ is of highest weight, and
  therefore equals the highest weight vertex of $B_\lambda$.
\end{proof}

\subsection{Promotion and evacuation of $\GL(n)$-alternating
  tableaux}

Let us now make promotion and evacuation of $\GL(n)$-alternating
tableaux explicit.  Recall from example~\ref{eg:alternating} that we
regard the adjoint representation as the tensor product
$V\otimes V^\ast$, where $V$ is the vector representation of $\GL(n)$
and $V^\ast$ is its dual.  Both of these are minuscule, so we can
apply Theorem~\ref{thm:local_commutor}.

Therefore, the rectangular grid to compute
$\pr w = \sigma_{C_1,C_2\otimes\dots\otimes C_r}(w)$ has three rows,
the bottom row begins with the zero weight, the middle row with
$\mu^1$, which is always $e_1$, and the top row contains the
remaining cumulative weights $\mu^2,\dots,\mu^{2r}$.

However, because we will apply promotion repeatedly, it will be
convenient to slightly enlarge this grid, and prepend the two weights
$\mu^0$ and $\mu^1$ to the first row.  After having successively
applied the local rule~\eqref{eq:local} and thus computed the
remaining weights in the middle and the bottom row, we append the
final element of the second row and the weight of the original word
to the third line.  We call the resulting diagram the \Dfn{promotion
  diagram} of an alternating tableau:
\begin{equation}\label{eq:pr-adjoint-schema}
\def\arraystretch{2.5}
\newcolumntype{C}{>{$}p{32pt}<{$}}
\newcolumntype{D}{>{$}p{53pt}<{$}}
\begin{array}{CCCCCDD}
    \wgt^0{\scriptstyle=\emptyset}&\wgt^1{=\scriptstyle1}&\wgt^2&\dotfill&\wgt^{2r}\\
    &&\hm^1{\scriptstyle=\wgt^1}&\dotfill&\hm^{2r-1}\\
    &&\Pm^0{=\scriptstyle\wgt^0}&\dotfill&\Pm^{2r-2}&\Pm^{2r-1}{=\scriptstyle\hm^{2r-1}}&\Pm^{2r}{=\scriptstyle\wgt^{2r}}
\end{array}
\end{equation}

To illustrate, let us compute the promotion of the $\GL(3)$ highest
weight word
$w=(e_1,\text{-}e_3)\,(e_1,\text{-}e_2)\,(e_2,\text{-}e_2)\,(e_2,\text{-}e_1)\,(e_3,\text{-}e_1)$.
The first row is the alternating tableau corresponding to $w$.  For
better readability we write $\bar 1$ in place of $-1$.
\begin{equation}\label{eq:pr-adjoint-example}
\def\arraystretch{1.5}
\begin{array}{ccccccccccccc}
  000&100&10\bar1&20\bar1&2\bar1\bar1&\tikzmark{start}20\bar1&2\bar1\bar1&20\bar1&10\bar1&100&000\\
     &   &100    &200    &20\bar1    &21\bar1            &20\bar1    &21\bar1&11\bar1&110&100\\
     &   &000    &100    &10\bar1    &11\bar1                &10\bar1\tikzmark{end}&11\bar1&10\bar1&100&10\bar1&100&000
\end{array}
\begin{tikzpicture}[remember picture,overlay]
  \node[draw,line width=0pt,rectangle,minimum size=1.8cm,fit={(pic cs:start) (pic cs:end)},yshift=3pt] {};
\end{tikzpicture}
\end{equation}
Thus, the promotion of $w$ is
$\pr
w=(e_1,\text{-}e_3)\,(e_2,\text{-}e_2)\,(e_2,\text{-}e_2)\,(e_3,\text{-}e_3)\,(e_3,\text{-}e_1)$.

The six vectors in the rectangle in
diagram~\eqref{eq:pr-adjoint-example} demonstrate that the naive
embedding of $\GL(n)$-alternating tableaux into the set of
$\GL(n+1)$-alternating tableaux is not compatible with promotion, as
already mentioned in Section~\ref{sec:adjoint-results}: padding the
vectors of the original word with zeros, and applying the local rule,
we obtain the rectangle
\[
\def\arraystretch{1.5}
\begin{array}{cc}
  200\bar1&20\bar1\bar1\\
  210\bar1&21\bar1\bar1\\
  110\bar1&11\bar1\bar1,
\end{array}
\]
with bottom-right vector $11\bar1\bar1$, rather than $100\bar1$ as
one might expect.

\begin{figure}[h]
\[
\def\arraystretch{1.75}
\setlength{\arraycolsep}{3pt}
\begin{array}{r*{14}{c}}
\alt=0 0 0&1 0 0&1 0\bar1&2 0\bar1&2 0\bar2&2 0\bar1&2\bar1\bar1&3\bar1\bar1&2\bar1\bar1\tikzmark{min1}&2 0\bar1&2 0\bar2&3 0\bar2&3 0\bar3&3 1\bar3&2 1\bar3\\
        &       &1 0 0&2 0 0&2 0\bar1&2 0 0&2 0\bar1&3 0\bar1&2 0\bar1&2 1\bar1&2 1\bar2&3 1\bar2&3 1\bar3&3 2\bar3&2 2\bar3\\
        &       &0 0 0&1 0 0&1 0\bar1\tikzmark{min2}&1 0 0&1 0\bar1&2 0\bar1&1 0\bar1&1 1\bar1&1 1\bar2&2 1\bar2&2 1\bar3&2 2\bar3&2 1\bar3\\
        &       &       &       &1 0 0&1 1 0&1 1\bar1&2 1\bar1&1 1\bar1&2 1\bar1&2 1\bar2&3 1\bar2&3 1\bar3&3 2\bar3&3 1\bar3\\
        &       &       &       &0 0 0&1 0 0&1 0\bar1&2 0\bar1&1 0\bar1&2 0\bar1&2 0\bar2&3 0\bar2&3 0\bar3&3 1\bar3&3 0\bar3\\
        &       &       &       &       &       &1 0 0&2 0 0&1 0 0&2 0 0&2 0\bar1&3 0\bar1&3 0\bar2&3 1\bar2&3 0\bar2\\
        &       &       &       &       &       &0 0 0&1 0 0\tikzmark{fixd}&0 0 0&1 0 0&1 0\bar1&2 0\bar1&2 0\bar2&2 1\bar2&2 0\bar2\\
        &       &       &       &       &       &       &       &1 0 0&2 0 0&2 0\bar1&3 0\bar1&3 0\bar2&3 1\bar2&3 0\bar2\\
        &       &       &       &       &       &       &       &0 0 0&1 0 0&1 0\bar1&2 0\bar1&2 0\bar2&2 1\bar2&2 0\bar2\\
        &       &       &       &       &       &       &       &       &       &1 0 0&2 0 0&2 0\bar1&2 1\bar1&2 0\bar1\\
        &       &       &       &       &       &       &       &       &       &0 0 0&1 0 0&1 0\bar1&1 1\bar1&1 0\bar1\\
        &       &       &       &       &       &       &       &       &       &       &       &1 0 0&1 1 0\tikzmark{plus}&1 0 0\\
        &       &       &       &       &       &       &       &       &       &       &       &0 0 0&1 0 0&1 0\bar1\\
        &       &       &       &       &       &       &       &       &       &       &       &       &       &1 0 0\\
        &       &       &       &       &       &       &       &       &       &       &       &       &       &0 0 0\\
&&&&&&&&&&&&&&\rotatebox{90}{$\ev\alt=$}
\end{array}
\]
\begin{tikzpicture}[remember picture,overlay]
  \node [xshift=3pt,yshift=-7pt]  at (pic cs:min1) {\negcross};
  \node [xshift=3pt,yshift=-7pt]  at (pic cs:min2) {\negcross};
  \node [xshift=3pt,yshift=-7pt]  at (pic cs:plus) {\poscross};
  \node [xshift=3pt,yshift=-7pt]  at (pic cs:fixd) {\fixcross};
\end{tikzpicture}
\caption{The evacuation of an alternating tableau.}
\label{fig:ev-adjoint-example}
\end{figure}
To obtain the evacuation of an alternating tableau we use the second
identity of Lemma~\ref{lem:extend-s_1q}.  We start by computing the
promotion of the initial alternating tableau as above, except that we
do not append anything to the third row.  We then repeat this
process a total of $r$ times, creating a (roughly) triangular array
of weights, which we call the \Dfn{evacuation diagram} of an alternating tableau.  The sequence of cumulative weights of the evacuation can
then be read off the vertical row on the right hand side, from
bottom to top.  An example can be found in Figure~\ref{fig:ev-adjoint-example}.  The symbols \poscross, \negcross\ and \fixcross\ occurring in the figure should be ignored for the moment.

Finally, we would like to point out that for alternating tableaux of
empty shape there is a second way to compute the promotion,
exploiting the fact that the next-to-last weight is forced to be
$10\dots0$.  Let $w=w_1\dots w_r$ be the highest weight word
corresponding to the alternating tableau.  We consider $w$ as an
element of $A\otimes A^\ast\otimes B_\lambda$, where $A$ is the
crystal corresponding to $V$, $A^\ast$ is the crystal corresponding
to $V^\ast$ and, similar to what was done in the proof of
Corollary~\ref{cor:sliding}, $B_\lambda$ is the component of
$\bigotimes^{r-1}(A\otimes A^\ast)$ containing $w_3\dots w_r$.  Then
we first compute $\hat w=\sigma_{A, A^\ast\otimes B_\lambda}(w)$,
followed by computing
$\hat{\hat w}=\sigma_{A^\ast,B_\lambda\otimes A}(\hat w)$:
\begin{equation}\label{eq:pr-stephan-adjoint-example}
\def\arraystretch{1.5}
\begin{array}{ccccccccccccc}
  000&100&10\bar1&20\bar1&2\bar1\bar1&20\bar1&2\bar1\bar1&20\bar1&10\bar1&100&000\\
  &000&00\bar1&10\bar1&1\bar1\bar1&10\bar1&1\bar1\bar1&10\bar1&00\bar1&000&00\bar1&000\\
  &   &000    &100    &10\bar1    &11\bar1                &10\bar1&11\bar1&10\bar1&100&10\bar1&100&000
\end{array}
\end{equation}
Because the initial segment of $\hat{\hat w}$ is an element of $B_\lambda$ and is of highest weight, it must coincide with the initial segment of the promotion of $w$.

This variant of the local rules for promotion was recently
rediscovered, in slightly different form, by
Patrias~\cite{MR3861768}.  Note, however, that for an alternating
tableau of non-empty shape, this procedure yields a tableau which, in
general, is different from the result of promotion.
\section{Growth diagram bijections}
\label{sec:growth-diagrams}

In this section we recall Sundaram's bijection (using Roby's
description~\cite{MR2716353} based on Fomin's growth
diagrams~\cite{MR1314558}) between oscillating tableaux and
matchings.  We also present a new bijection, in the same spirit,
between alternating tableaux and partial permutations.  In both
cases, the action of the cactus group on highest weight words becomes
particularly transparent when using Fomin's growth diagrams and local
rules for the Robinson-Schensted correspondence.

\begin{figure}[h]
  \def\arraystretch{1.5}
  \begin{tabular}{rccc}
    &
      \begin{tikzpicture}[baseline={([yshift=-.5ex]current bounding box.center)},scale=2]
        \node (ka) at (0,1) {$\kappa$};
        \node (la) at (0,0) {$\lambda$};
        \node (nu) at (1,0) {$\nu$};
        \node (mu) at (1,1) {$\mu$};
        \draw[<-] (ka) -- (la);
        \draw[<-] (mu) -- (nu);
        \draw[->] (la) -- (nu);
        \draw[->] (ka) -- (mu);
        \node at (0.5,0.5)[rectangle,minimum size=1.5cm,draw] {};
      \end{tikzpicture}
    &or 
    &
      \begin{tikzpicture}[baseline={([yshift=-.5ex]current bounding box.center)},scale=2]
        \node (ka) at (0,1) {$\lambda$};
        \node (la) at (0,0) {$\lambda$};
        \node (nu) at (1,0) {$\lambda$};
        \node (mu) at (1,1) {$\mu$};
        \draw[<-] (ka) -- (la);
        \draw[<-] (mu) -- (nu);
        \draw[->] (la) -- (nu);
        \draw[->] (ka) -- (mu);
        \node at (0.5,0.5)[rectangle,minimum size=1.5cm,draw] {\huge$\times$};
      \end{tikzpicture}\\
    forward rules: %
    & $\mu'=\domS(\kappa'+\nu'-\lambda')$ & %
    & $\mu = \lambda + e_1$\\
    backward rules: %
    & $\lambda'=\domS(\kappa'+\nu'-\mu')$ & %
    & $\lambda = \mu - e_1$ %
  \end{tabular}
  \caption{Cells of a growth diagram and corresponding local rules.}
  \label{fig:growth-cell}
\end{figure}
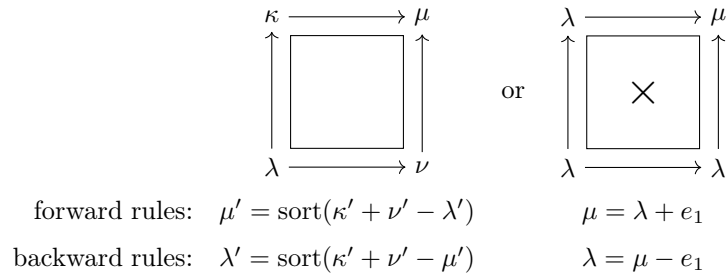
For our purposes, a growth diagram is a finite collection of cells,
arranged in the form of a Ferrers diagram using the French
convention, as for example in Figures~\ref{fig:s1q-partial-matchings}
and ~\ref{fig:s1q-partial-permutations}.  Let us first describe the
classical setup, which we use to describe Sundaram's correspondence.

In this case, each cell is either empty or contains a cross.
Moreover, we require that in every row and every column of the growth
diagram there is at most one cell which contains a cross.  Every
corner of a cell is labelled with a partition such that the local
rules in Figure~\ref{fig:growth-cell} are satisfied, where
$\lambda^\prime$ denotes the partition conjugate to~$\lambda$
and~$e_1$ is the first unit vector.  Moreover, we require that two
adjacent partitions (as for example~$\lambda$ and~$\kappa$ in
Figure~\ref{fig:growth-cell}) either coincide or the one at the head
of the arrow is obtained from the other by adding a unit vector.

Furthermore, we require that the partitions labelling the corners of
a cell satisfy the forward and backward rules of
Figure~\ref{fig:growth-cell}.  In fact, the two \Dfn{forward rules}
determine~$\mu$ given the other three partitions and the content of
the cell.  The two \Dfn{backward rules} determine the content of the
cell and the partition~$\lambda$ in the bottom-left given the other
three partition.

Thus, the information in a growth diagram is redundant.  In
particular, given the partitions labelling the corners along the
bottom-left border and the contents of the cells, one can recover the
remaining partitions.  Conversely, given the partitions labelling the
corners along the top-right border of a diagram, one can recover the
remaining partitions and the contents of the cells.

The presentation of the local rules in Figure~\ref{fig:growth-cell}
is slightly non-standard.  It has the benefit that the local rule for
empty cells is very similar to the special case of
Definition~\ref{def:local-rules} corresponding to Cartan type~$A_n$,
with two important differences.  The first difference is that all
partitions are transposed, the second, that the orientation of the
vertical arrows is reversed.

\subsection{Roby's description of Sundaram's correspondence}
\label{sec:Roby}
\begin{figure}
  \begin{center}
\begin{tikzpicture}[scale=0.75]\tiny
% a dotted line to indicate the reflection
\draw[dotted] (1,1) -- (5,5);
\foreach \y in {1,...,10}
   \draw (1,11-\y)--(\y,11-\y);
\foreach \x in {1,2,...,10}
   \draw (\x,11-\x)--(\x,1);
% column labels
\foreach \x in {1,2,...,9}
   \draw (\x.5,11-10.3) node{\footnotesize$\x$};
% row labels
\foreach \x in {1,2,...,9}
   \draw (0.3,11-\x.4) node{\footnotesize$\x$};
% crosses
\draw (1.5,11-4.5) node{\huge$\times$};
\draw (2.5,11-9.5) node{\huge$\times$};
\draw (3.5,11-6.5) node{\huge$\times$};
% partitions column 0 (from bottom to top)
\draw (.8,11-.8) node{\footnotesize$\emptyset$};
\foreach \x in {1,2,...,8}
   \draw (.8,11-\x.8) node{$\emptyset$};
\draw (0.8,11-9.8) node{\footnotesize$\emptyset$};
% partitions column 1
\draw (1.8,11-1.8) node{\footnotesize$1$};
\draw (1.8,11-2.8) node{$1$};
\draw (1.8,11-3.8) node{$1$};
\foreach \x in {4,5,...,8}
   \draw (1.8,11-\x.8) node{$\emptyset$};
\draw (1.8,11-9.8) node{\footnotesize$\emptyset$};
% partitions column 2
\draw (2.8,11-2.8) node{\footnotesize$11$};
\draw (2.8,11-3.8) node{$11$};
\foreach \x in {4,5,...,8}
   \draw (2.8,11-\x.8) node{$1$};
\draw (2.8,11-9.8) node{\footnotesize$\emptyset$};
% partitions column 3
\draw (3.8,11-3.8) node{\footnotesize$21$};
\draw (3.8,11-4.8) node{$2$};
\draw (3.8,11-5.8) node{$2$};
\draw (3.8,11-6.8) node{$1$};
\draw (3.8,11-7.8) node{$1$};
\draw (3.8,11-8.8) node{$1$};
\draw (3.8,11-9.8) node{\footnotesize$\emptyset$};
% partitions column 4
\draw (4.8,11-4.8) node{\footnotesize$2$};
\draw (4.8,11-5.8) node{$2$};
\draw (4.8,11-6.8) node{$1$};
\draw (4.8,11-7.8) node{$1$};
\draw (4.8,11-8.8) node{$1$};
\draw (4.8,11-9.8) node{\footnotesize$\emptyset$};
% partitions column 5
\draw (5.8,11-5.8) node{\footnotesize$21$};
\draw (5.8,11-6.8) node{$11$};
\draw (5.8,11-7.8) node{$11$};
\draw (5.8,11-8.8) node{$11$};
\draw (5.8,11-9.8) node{\footnotesize$1$};
% partitions column 6
\draw (6.8,11-6.8) node{\footnotesize$11$};
\draw (6.8,11-7.8) node{$11$};
\draw (6.8,11-8.8) node{$11$};
\draw (6.8,11-9.8) node{\footnotesize$1$};
% partitions column 7
\draw (7.8,11-7.8) node{\footnotesize$21$};
\draw (7.8,11-8.8) node{$21$};
\draw (7.8,11-9.8) node{\footnotesize$2$};
% partitions column 8
\draw (8.8,11-8.8) node{\footnotesize$211$};
\draw (8.8,11-9.8) node{\footnotesize$21$};
% partitions column 9
\draw (9.8,11-9.8) node{\footnotesize$21$};
\end{tikzpicture}
\begin{tikzpicture}[scale=0.75]\tiny
% a dotted line to indicate the reflection
\draw[dotted] (1,1) -- (5,5);
\foreach \y in {1,...,10}
   \draw (1,11-\y)--(\y,11-\y);
\foreach \x in {1,2,...,10}
   \draw (\x,11-\x)--(\x,1);
% column labels
\foreach \x in {1,2,...,9}
   \draw (\x.5,11-10.3) node{\footnotesize$\x$};
% row labels
\foreach \x in {1,2,...,9}
   \draw (0.3,11-\x.4) node{\footnotesize$\x$};
% crosses
\draw (1.5,11-8.5) node{\huge$\times$};
\draw (4.5,11-7.5) node{\huge$\times$};
\draw (6.5,11-9.5) node{\huge$\times$};
% partitions column 0 (from bottom to top)
\draw (.8,11-.8) node{\footnotesize$\emptyset$};
\foreach \x in {1,2,...,8}
   \draw (.8,11-\x.8) node{$\emptyset$};
\draw (0.8,11-9.8) node{\footnotesize$\emptyset$};
% partitions column 1
\draw (1.8,11-1.8) node{\footnotesize$1$};
\foreach \x in {2,3,...,7}
   \draw (1.8,11-\x.8) node{$1$};
\draw (1.8,11-8.8) node{$\emptyset$};
\draw (1.8,11-9.8) node{\footnotesize$\emptyset$};
% partitions column 2
\draw (2.8,11-2.8) node{\footnotesize$11$};
\foreach \x in {3,4,...,7}
   \draw (2.8,11-\x.8) node{$11$};
\draw (2.8,11-8.8) node{$1$};
\draw (2.8,11-9.8) node{\footnotesize$1$};
% partitions column 3
\draw (3.8,11-3.8) node{\footnotesize$111$};
\foreach \x in {4,5,6,7}
   \draw (3.8,11-\x.8) node{$111$};
\draw (3.8,11-8.8) node{$11$};
\draw (3.8,11-9.8) node{\footnotesize$11$};
% partitions column 4
\draw (4.8,11-4.8) node{\footnotesize$211$};
\foreach \x in {5,6}
   \draw (4.8,11-\x.8) node{$211$};
\draw (4.8,11-7.8) node{$111$};
\draw (4.8,11-8.8) node{$11$};
\draw (4.8,11-9.8) node{\footnotesize$11$};
% partitions column 5
\draw (5.8,11-5.8) node{\footnotesize$221$};
\draw (5.8,11-6.8) node{$221$};
\draw (5.8,11-7.8) node{$211$};
\draw (5.8,11-8.8) node{$21$};
\draw (5.8,11-9.8) node{\footnotesize$21$};
% partitions column 6
\draw (6.8,11-6.8) node{\footnotesize$321$};
\draw (6.8,11-7.8) node{$311$};
\draw (6.8,11-8.8) node{$31$};
\draw (6.8,11-9.8) node{\footnotesize$21$};
% partitions column 7
\draw (7.8,11-7.8) node{\footnotesize$311$};
\draw (7.8,11-8.8) node{$31$};
\draw (7.8,11-9.8) node{\footnotesize$21$};
% partitions column 8
\draw (8.8,11-8.8) node{\footnotesize$31$};
\draw (8.8,11-9.8) node{\footnotesize$21$};
% partitions column 9
\draw (9.8,11-9.8) node{\footnotesize$21$};
\end{tikzpicture}
  \end{center}
  \caption{A pair of growth diagrams $\G(\osc)$ and
    $\G(\cactus_{1,9}\osc)$, with $\osc=(\emptyset,1,11,21,2,21,11,21,211,21)$, illustrating
    Theorem~\ref{thm:s1q-partial-matchings}.  The dotted line
    indicates the axis of reflection for the matchings $\M(\osc)$ and
    $\M(\cactus_{1,9}\osc)$.}
  \label{fig:s1q-partial-matchings}
\end{figure}
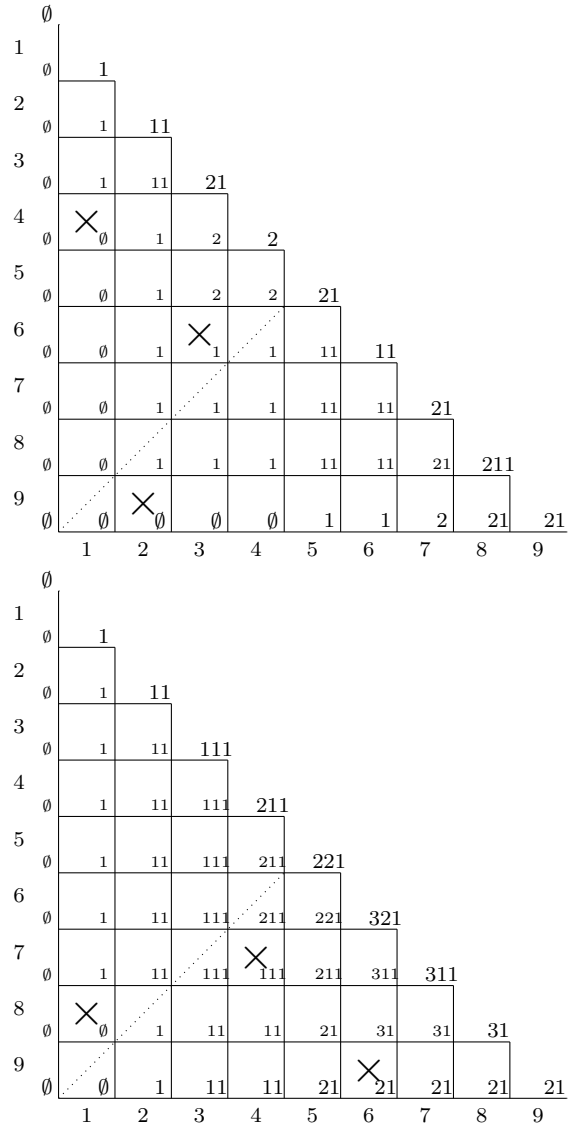

In this section we recall the bijection between oscillating tableaux
of length $r$ and shape $\mu$ and pairs consisting of a partial
matching of $\{1,\dots,r\}$ and a partial standard Young tableau of
shape $\mu$ whose entries are the unmatched elements.

\begin{defn}
  Let $\osc=(\wgt_0,\wgt_1,\ldots,\wgt_r)$ be an oscillating
  tableau.  The associated (triangular) growth diagram $\G(\osc)$
  consists of $r$ left-justified rows, with $i-1$ cells in row $i$
  for $i\in\{1,\dots,r\}$, where row $1$ is the top row.  Label
  the cells according to the following specification:
  \begin{itemize}\setlength{\itemsep}{0pt}
  \item[R1] Label the north east corners of the cells on the main diagonal from
    the top-left to the bottom-right with the partitions in $\osc$.
  \item[R2] Label the corners of the first subdiagonal with the
    smaller of the two partitions labelling the two adjacent corners
    on the diagonal.
  \item[R3] Use the backward rules to determine which cells contain a
    cross.
  \end{itemize}

  Let $\M(\osc)$ be the matching containing a pair $\{i, j\}$ for
  every cross in column~$i$ and row~$j$ of the $\G(\osc)$.
  Furthermore, let $\M_T(\osc)$ be the partial standard Young tableau
  corresponding to the sequence of partitions along the bottom border
  of $\G(\osc)$.
\end{defn}

\begin{thm}[\protect{Sundaram~\cite[Sec.~8]{MR2941115},
    Roby~\cite[Prop.~4.3.1]{MR2716353}}]
  The map $\osc\mapsto\big(\M(\osc), \M_T(\osc)\big)$ is a bijection
  between oscillating tableaux of length $r$ and shape $\mu$, and
  pairs consisting of a perfect matching of a subset of
  $\{1,\dots,r\}$ and a partial standard Young tableau of shape
  $\mu$, whose entries form the complementary subset.

  Moreover, the map $\osc\mapsto\M(\osc)$ is a bijection between
  $n$-symplectic oscillating tableaux of length $r$ and empty shape
  and $(n+1)$-noncrossing perfect matchings of~$\{1,\dots,r\}$.
\end{thm}

An example for this procedure, which also illustrates
Theorem~\ref{thm:s1q-partial-matchings}, can be found in
Figure~\ref{fig:s1q-partial-matchings}.  Let $\osc$ be the
$3$-symplectic oscillating tableau
\[
(\emptyset,1,11,21,2,21,11,21,211,21),
\]
whose partitions label the
corners of the diagonal of the first growth diagram.  Applying the
backward rules, we obtain the matching and the partial standard Young
tableau \Yvcentermath1
\[
\M(\osc)=\big\{\{1,4\},\{2,9\},\{3,6\}\big\}\text{ and }
\M_T(\osc)=\young(57,8).
\]
Using Lemma~\ref{lem:extend-s_1q} and the local rule in
Definition~\ref{def:local-rules} one can compute that $\ev\osc$ is
the $3$-symplectic oscillating tableau labelling the corners of the
diagonal of the second growth diagram.  Applying the backward rules
again, we obtain the matching and the partial standard Young tableau
predicted by Theorem~\ref{thm:s1q-partial-matchings}: \Yvcentermath1
\[
\M(\ev\osc)=\big\{\{1,8\},\{4,7\},\{6,9\}\big\}\text{ and }
\M_T(\ev\osc)=\young(25,3).
\]

\subsection{A new variant for Stembridge's alternating tableaux}
\label{sec:Stembridge}

In this section we present a variation of Sundaram's bijection for
alternating tableaux and permutations.

Recall that a staircase is a vector with weakly decreasing integer
entries. The \Dfn{positive part} of the staircase is the partition
obtained by removing all entries less than or equal to zero. The
\Dfn{negative part} of the staircase is the partition obtained by
removing all entries greater than or equal to zero, removing the
signs of the remaining entries and reversing the sequence.
\begin{defn}\label{defn:P}
  Let $\alt=(\wgt^0,\wgt^1,\ldots,\wgt^{2r})$ be an alternating
  tableau.  The associated growth diagram $\G(\alt)$ is an
  $r\times r$ square of cells, obtained as follows:
  \begin{itemize}\setlength{\itemsep}{0pt}
  \item[P1] Label the north east corners of the cells on the main diagonal and
    the first superdiagonal from the top-left to the bottom-right
    with the staircases in $\alt$.
  \item[P2] Apply the backward rules on the positive parts of the
    staircases to determine which cells below the diagonal contain a
    cross.
  \item[P3] Use the backward rules (rotated by $180\degree$) on the
    negative parts of the staircases to determine which cells above
    the diagonal contain a cross.
  \end{itemize}

  Let $\Perm(\alt)$ be the partial permutation mapping $i$ to $j$ for
  every cross in column $i$ and row $j$ of $\G(\alt)$, and let
  $\big(\Perm_P(\alt), \Perm_Q(\alt)\big)$ be the pair of partial
  standard Young tableaux corresponding to the sequence of partitions
  along the bottom and the right border of $\G(\alt)$, respectively.
\end{defn}

\begin{thm}
  The map
  $\alt\mapsto\big(\Perm(\alt), \Perm_P(\alt), \Perm_Q(\alt)\big)$ is
  a bijection between alternating tableaux of length $r$ and shape
  $\mu$, and triples consisting of a bijection $\Perm(\alt):R\to S$
  between two subsets of $\{1,\dots,r\}$, and two partial standard
  Young tableaux $\Perm_P(\alt)$ and $\Perm_Q(\alt)$.  The shapes of
  these tableaux are obtained by separating the positive and negative
  entries of $\mu$.  The entries of the first tableau then form the
  complementary subset of $R$, the entries of the second form the
  complementary subset of $S$.

  Moreover, the map $\alt\mapsto\Perm(\alt)$ is a bijection between
  $\GL(n)$-alternating tableaux of length $r$ and empty shape and
  permutations of~$\{1,\dots,r\}$ whose longest increasing
  subsequence has length at most $n$.
\end{thm}

An example for this procedure, which also illustrates
Theorem~\ref{thm:s1q-partial-permutations}, can be found in
Figure~\ref{fig:s1q-partial-permutations}.  We render fixed points as
\fixcross, other crosses below the diagonal as \poscross\ and crosses
above the diagonal as \negcross.  The reason for doing so is given by
Corollary~\ref{cor:relate-fillings} in Section~\ref{sec:alternating},
where we show that the growth diagram of an alternating tableau and
its evacuation diagram are very closely related.

Let $\alt$ be the $\GL(13)$-alternating tableau of length $7$
\[
(\emptyset, 1, 1\bar1, 2\bar1, 2\bar2, 2\bar1, 2\bar1\bar1,
3\bar1\bar1, 2\bar1\bar1, 2\bar1, 2\bar2, 3\bar2, 3\bar3, 31\bar3,
21\bar3),
\]
where we write the negative entries with bars and omit zeros for
readability.  Its staircases label the corners of the diagonal of the
first growth diagram.  Applying the backward rules we obtain the
partial permutation and the partial standard Young tableaux
\Yvcentermath1
\[
\Perm(\alt)=\big\{(3,2),(4,4),(5,1),(6,7)\big\},
\Perm_P(\alt)=\young(356),\text{ and } \Perm_Q(\alt)=\young(12,7).
\]
The second growth diagram in the figure is obtained by applying the
same procedure to $\ev\alt =\cactus_{1,7}\alt$, which yields 
\Yvcentermath1
\[
\Perm(\ev\alt)=\big\{(2,1),(3,7),(4,4),(5,6)\big\},
\Perm_P(\ev\alt)=\young(235),\text{ and }
\Perm_Q(\ev\alt)=\young(17,6).
\]
as predicted by Theorem~\ref{thm:s1q-partial-permutations}.

In the example above, we could have obtained the same sequence of
positive and negative parts of the staircases from a
$\GL(3)$-alternating tableau, removing ten zeros from each vector.
As it turns out, evacuation of this alternating tableau yields the
same result as above, although for $r=7$
Theorem~\ref{thm:s1q-partial-permutations} applies only when $n$ is
at least~$13$. The computation of the evacuated alternating tableau
is carried out in Figure~\ref{fig:ev-adjoint-example}.

The $\GL(2)$-alternating tableau
$\alt=(\emptyset, 10, 1\bar1, 10, 1\bar1)$ of length $2$
%$(e_1,\text{-}e_2)\;(e_2,\text{-}e_2)$ 
illustrates the necessity of the hypothesis restricting the length of
the alternating tableau in
Theorem~\ref{thm:s1q-partial-permutations}.  On the one hand, this
tableau is fixed by $\ev=\cactus_{1,2}$.  On the other hand,
$\Perm(\alt) = \{(2,1)\}$, whose reverse-complement is $\{(1,2)\}$.

Similarly, to justify the necessity of the hypothesis in Theorem~\ref{thm:rotation-promotion-permutations}, consider the $\GL(3)$-alternating tableau in the first row of diagram~\eqref{eq:pr-adjoint-example}, which corresponds to the permutation depicted in Figure~\ref{fig:chord-diagrams}.  Its promotion, as computed in the last row of diagram~\eqref{eq:pr-adjoint-example}, corresponds to the permutation $23514$, which differs from the rotated permutation.

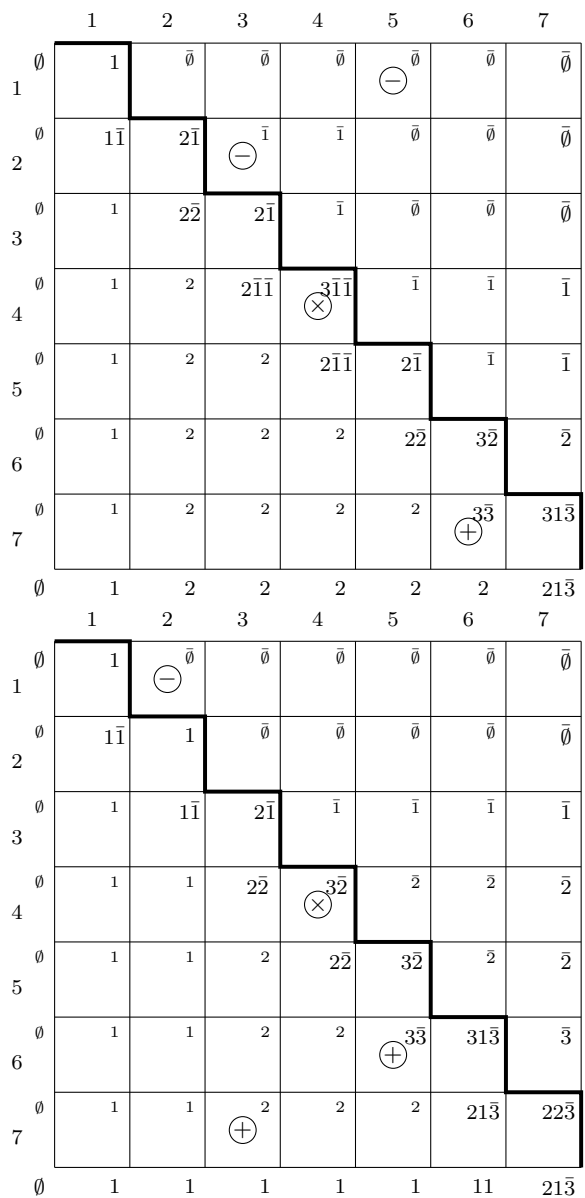
\begin{figure}
  \begin{center}
\begin{tikzpicture}[scale=1]\tiny
% thick diagonal
\draw[line width=1.5pt] (8,1) -- (8,2) -- (7,2) -- (7,3) -- (6,3) -- (6,4) --
(5,4) -- (5,5) -- (4,5) -- (4,6) -- (3,6) -- (3,7) -- (2,7) -- (2,8) -- (1,8);
\foreach \y in {1,...,8}
   \draw (1,\y)--(8,\y);
\foreach \x in {1,2,...,8}
   \draw (\x,1)--(\x,8);
% column labels
\foreach \x in {1,...,7}
   \draw (\x.5,8.3) node{\footnotesize$\x$};
% row labels
\foreach \x in {1,...,7}
   \draw (0.5,8.4-\x) node{\footnotesize$\x$};
% crosses
\draw (5.5,7.5) node{\negcross};
\draw (3.5,6.5) node{\negcross};
\draw (4.5,4.5) node{\fixcross};
\draw (6.5,1.5) node{\poscross};
% partitions column 0 (from bottom to top)
\draw (.8,.75) node{\footnotesize$\emptyset$};
\foreach \x in {1,...,6}
   \draw (.8,\x.8) node{$\emptyset$};
\draw (0.8,7.75) node{\footnotesize$\emptyset$};
% partitions column 1
\draw (1.8,0.75) node{\footnotesize$1$};
\foreach \x in {1,...,5}
   \draw (1.8,\x.8) node{$1$};
\draw (1.8,6.75) node{\footnotesize$1\bar1$};
\draw (1.8,7.75) node{\footnotesize$1$};
% partitions column 2
\draw (2.8,0.75) node{\footnotesize$2$};
\foreach \x in {1,...,4}
   \draw (2.8,\x.8) node{$2$};
\draw (2.8,5.75) node{\footnotesize$2\bar2$};
\draw (2.8,6.75) node{\footnotesize$2\bar1$};
\draw (2.8,7.8) node{$\bar\emptyset$};
% partitions column 3
\draw (3.8,0.75) node{\footnotesize$2$};
\foreach \x in {1,2,3}
   \draw (3.8,\x.8) node{$2$};
\draw (3.7,4.75) node{\footnotesize$2\bar1\bar1$};
\draw (3.8,5.75) node{\footnotesize$2\bar1$};
\draw (3.8,6.8) node{$\bar 1$};
\draw (3.8,7.8) node{$\bar\emptyset$};
% partitions column 4
\draw (4.8,0.75) node{\footnotesize$2$};
\draw (4.8,1.8) node{$2$};
\draw (4.8,2.8) node{$2$};
\draw (4.75,3.75) node{\footnotesize$2\bar1\bar1$};
\draw (4.75,4.75) node{\footnotesize$3\bar1\bar1$};
\foreach \x in {5,6}
   \draw (4.8,\x.8) node{$\bar1$};
\draw (4.8,7.8) node{$\bar\emptyset$};
% partitions column 5
\draw (5.8,0.75) node{\footnotesize$2$};
\draw (5.8,1.8) node{$2$};
\draw (5.8,2.75) node{\footnotesize$2\bar2$};
\draw (5.75,3.75) node{\footnotesize$2\bar1$};
\draw (5.8,4.8) node{$\bar1$};
\foreach \x in {5,6,7}
   \draw (5.8,\x.8) node{$\bar\emptyset$};
% partitions column 6
\draw (6.7,0.75) node{\footnotesize$2$};
\draw (6.7,1.75) node{\footnotesize$3\bar3$};
\draw (6.75,2.75) node{\footnotesize$3\bar2$};
\draw (6.8,3.8) node{$\bar1$};
\draw (6.8,4.8) node{$\bar1$};
\draw (6.8,5.8) node{$\bar\emptyset$};
\draw (6.8,6.8) node{$\bar\emptyset$};
\draw (6.8,7.8) node{$\bar\emptyset$};
% partitions column 7
\draw (7.7,0.75) node{\footnotesize$21\bar3$};
\draw (7.7,1.75) node{\footnotesize$31\bar3$};
\draw (7.8,2.75) node{\footnotesize$\bar2$};
\draw (7.8,3.75) node{\footnotesize$\bar1$};
\draw (7.8,4.75) node{\footnotesize$\bar1$};
\draw (7.8,5.75) node{\footnotesize$\bar\emptyset$};
\draw (7.8,6.75) node{\footnotesize$\bar\emptyset$};
\draw (7.8,7.75) node{\footnotesize$\bar\emptyset$};
\end{tikzpicture}
\hfil
\begin{tikzpicture}[scale=1]\tiny
% thick diagonal
\draw[line width=1.5pt] (8,1) -- (8,2) -- (7,2) -- (7,3) -- (6,3) -- (6,4) --
(5,4) -- (5,5) -- (4,5) -- (4,6) -- (3,6) -- (3,7) -- (2,7) -- (2,8) -- (1,8);
\foreach \y in {1,...,8}
   \draw (1,\y)--(8,\y);
\foreach \x in {1,2,...,8}
   \draw (\x,1)--(\x,8);
% column labels
\foreach \x in {1,...,7}
   \draw (\x.5,8.3) node{\footnotesize$\x$};
% row labels
\foreach \x in {1,...,7}
   \draw (0.5,8.4-\x) node{\footnotesize$\x$};
% crosses
\draw (9-5.5,9-7.5) node{\poscross};
\draw (9-3.5,9-6.5) node{\poscross};
\draw (9-4.5,9-4.5) node{\fixcross};
\draw (9-6.5,9-1.5) node{\negcross};
% partitions column 0 (from bottom to top)
\draw (.8,.75) node{\footnotesize$\emptyset$};
\foreach \x in {1,...,6}
   \draw (.8,\x.8) node{$\emptyset$};
\draw (0.8,7.75) node{\footnotesize$\emptyset$};
% partitions column 1
\draw (1.8,0.75) node{\footnotesize$1$};
\foreach \x in {1,...,5}
   \draw (1.8,\x.8) node{$1$};
\draw (1.8,6.75) node{\footnotesize$1\bar1$};
\draw (1.8,7.75) node{\footnotesize$1$};
% partitions column 2
\draw (2.8,0.75) node{\footnotesize$1$};
\foreach \x in {1,...,4}
   \draw (2.8,\x.8) node{$1$};
\draw (2.8,5.75) node{\footnotesize$1\bar1$};
\draw (2.8,6.75) node{\footnotesize$1$};
\draw (2.8,7.8) node{$\bar\emptyset$};
% partitions column 3
\draw (3.8,0.75) node{\footnotesize$1$};
\foreach \x in {1,2,3}
   \draw (3.8,\x.8) node{$2$};
\draw (3.7,4.75) node{\footnotesize$2\bar2$};
\draw (3.8,5.75) node{\footnotesize$2\bar1$};
\draw (3.8,6.8) node{$\bar\emptyset$};
\draw (3.8,7.8) node{$\bar\emptyset$};
% partitions column 4
\draw (4.8,0.75) node{\footnotesize$1$};
\draw (4.8,1.8) node{$2$};
\draw (4.8,2.8) node{$2$};
\draw (4.8,3.75) node{\footnotesize$2\bar2$};
\draw (4.75,4.75) node{\footnotesize$3\bar2$};
\draw (4.75,5.8) node{$\bar1$};
\foreach \x in {6,7}
   \draw (4.8,\x.8) node{$\bar\emptyset$};
% partitions column 5
\draw (5.8,0.75) node{\footnotesize$1$};
\draw (5.8,1.8) node{$2$};
\draw (5.8,2.75) node{\footnotesize$3\bar3$};
\draw (5.75,3.75) node{\footnotesize$3\bar2$};
\draw (5.8,4.8) node{$\bar2$};
\draw (5.8,5.8) node{$\bar1$};
\foreach \x in {6,7}
   \draw (5.8,\x.8) node{$\bar\emptyset$};
% partitions column 6
\draw (6.7,0.75) node{\footnotesize$11$};
\draw (6.7,1.75) node{\footnotesize$21\bar3$};
\draw (6.7,2.75) node{\footnotesize$31\bar3$};
\draw (6.8,3.8) node{$\bar2$};
\draw (6.8,4.8) node{$\bar2$};
\draw (6.8,5.8) node{$\bar1$};
\draw (6.8,6.8) node{$\bar\emptyset$};
\draw (6.8,7.8) node{$\bar\emptyset$};
% partitions column 7
\draw (7.7,0.75) node{\footnotesize$21\bar3$};
\draw (7.7,1.75) node{\footnotesize$22\bar3$};
\draw (7.8,2.75) node{\footnotesize$\bar3$};
\draw (7.8,3.75) node{\footnotesize$\bar2$};
\draw (7.8,4.75) node{\footnotesize$\bar2$};
\draw (7.8,5.75) node{\footnotesize$\bar1$};
\draw (7.8,6.75) node{\footnotesize$\bar\emptyset$};
\draw (7.8,7.75) node{\footnotesize$\bar\emptyset$};
\end{tikzpicture}
  \end{center}
  \caption{A pair of growth diagrams $\G(\alt)$ and
    $\G(\cactus_{1,7}\alt)$, with $\alt=(\emptyset, 1, 1\bar1, 2\bar1, 2\bar2, 2\bar1, 2\bar1\bar1,
3\bar1\bar1, 2\bar1\bar1, 2\bar1, 2\bar2, 3\bar2, 3\bar3, 31\bar3,
21\bar3)$, illustrating
    Theorem~\ref{thm:s1q-partial-permutations}.}
  \label{fig:s1q-partial-permutations}
\end{figure}

\section{Proofs}\label{sec:proofs}

Our strategy is as follows.  We first consider only $\GL(n)$-alternating tableaux of empty shape and length $r$ with $n\geq r$, and show that the bijection $\Perm$ presented in Section~\ref{sec:Stembridge} intertwines rotation and promotion.  To do so, we demonstrate  that the middle row of the promotion diagram~\eqref{eq:pr-adjoint-schema} of an alternating tableau $\alt$ can be interpreted as corresponding to a single-step rotation of the rows of the growth diagram $\G(\alt)$. Then, using a very similar argument, we find that the promotion of $\alt$ corresponds to a single-step rotation of the columns of the growth diagram just obtained.

To prove the statements concerning evacuation, we show that the permutation $\Perm(\alt)$ can actually be read off directly from the evacuation diagram. In particular, this makes the effect of evacuation on $\Perm(\alt)$ completely transparent. The effect of the evacuation of an arbitrary alternating tableau $\alt$ on the triple $\big(\Perm(\alt),\Perm_P(\alt),\Perm_Q(\alt)\big)$ is deduced from the special case of alternating tableaux of empty shape by extending $\alt$ to an alternating tableau of empty shape.

In order to determine the exact range of validity of Theorem~\ref{thm:rotation-promotion-permutations} we use a stability phenomenon proved in Section~\ref{sec:stability}.  The case $n=2$ is treated completely separately in Section~\ref{sec:GL2}.

Finally, in Section~\ref{sec:oscillating}, we deduce the statements for oscillating tableaux and the vector representation of the symplectic groups, Theorem~\ref{thm:s1q-partial-matchings} and~\ref{thm:rotation-promotion-matchings}, from the statements for alternating tableaux.

\subsection{Stability}\label{sec:stability}
In this section we prove a stability phenomenon needed for
establishing the exact bounds in
Theorem~\ref{thm:rotation-promotion-permutations}.  Given the lack of
an embedding of $\GL(n)$-alternating tableaux in the set of
$\GL(n+1)$-alternating tableaux that is compatible with promotion,
this theorem may be interesting in its own right.
\begin{thm}\label{thm:stability}
    Let $\alt$ be a $\GL(n)$-alternating tableau, not necessarily of empty shape, and suppose that each staircase in $\alt$ and $\pr\alt$ contains at most $m$ nonzero parts.

    Then $\pr\tilde\alt=\widetilde{\pr\alt}$, where $\tilde\alt$ and $\widetilde{\pr\alt}$ are the $\GL(m)$-alternating tableaux obtained from $\alt$ and $\pr\alt$ by removing $n-m$ zeros from each staircase.
\end{thm}
Before proceeding to the proof, let us remark that this is not a trivial statement: it may well be that some staircases in the intermediate row $\halfpr\alt$ have more than $m$ nonzero parts.
\begin{proof}
  It suffices to consider the case $m=n-1$. We show inductively that
  the statement is true for every square of staircases in
  diagram~\eqref{eq:pr-adjoint-schema}
    \begin{equation}\label{eq:square-of-staircases}
      \tikzset{ short/.style={ shorten >=7pt, shorten <=7pt } }
      \begin{tikzpicture}[xscale=2.5,yscale=1.5,baseline=(current bounding box.center)]
        \node (ka) at (0,0) {$\kappa=\Pm^{2i-4}$};
        \node (be) at (0,1) {$\beta=\hm^{2i-3}$};
        \node (la) at (0,2) {$\lambda=\wgt^{2i-2}$};
        \node (de) at (1,0) {$\delta=\Pm^{2i-3}$};
        \node (ep) at (1,1) {$\varepsilon=\hm^{2i-2}$};
        \node (al) at (1,2) {$\alpha=\wgt^{2i-1}$};
        \node (mu) at (2,0) {$\mu=\Pm^{2i-2}$};
        \node (ga) at (2,1) {$\gamma=\hm^{2i-1}$};
        \node (nu) at (2,2) {$\nu=\wgt^{2i}$};
        % vertical
        \draw[->] (ka) -- (be) node [midway, left] {$+$};
        \draw[->] (be) -- (la) node [midway, left] {$-$};
        \draw[->] (de) -- (ep) node [midway, left] {$+$};
        \draw[->] (ep) -- (al) node [midway, left] {$-$};
        \draw[->] (mu) -- (ga) node [midway, left] {$+$};
        \draw[->] (ga) -- (nu) node [midway, left] {$-$};
        % horizontal
        \draw[->] (ka) -- (de) node [midway, above] {$+$};
        \draw[->] (de) -- (mu) node [midway, above] {$-$};
        \draw[->] (be) -- (ep) node [midway, above] {$+$};
        \draw[->] (ep) -- (ga) node [midway, above] {$-$};
        \draw[->] (la) -- (al) node [midway, above] {$+$};
        \draw[->] (al) -- (nu) node [midway, above] {$-$};
      \end{tikzpicture}
    \end{equation}
    where a $+$ between two staircases indicates that a unit vector
    is added to the staircase on the left (respectively, in the lower
    row) to obtain the staircase on the right (respectively, in the
    upper row).

    By assumption, all staircases in the top and bottom row contain
    at least one zero entry.  For such a staircase $\rho\in\ZZ^n$,
    let $\tilde\rho\in\ZZ^{n-1}$ be the staircase obtained from
    $\rho$ by removing a zero entry.  If $\rho$ does not contain a
    zero, it must contain an entry $1$ (say, at position $i$),
    followed by a negative entry.  In this case,
    $\tilde\rho\in\ZZ^{n-1}$ is obtained from $\rho$ by removing
    $\rho_i$ and adding $1$ to $\rho_{i+1}$.

    With this notation, we have to show the following four equalities
\begin{enumerate}[(a)]
\item\label{eq:stab-a} $\tilde\varepsilon = \dom_{\fS_{n-1}}(\tilde\beta + \tilde\alpha-\tilde\lambda)$,
\item\label{eq:stab-b} $\tilde\gamma = \dom_{\fS_{n-1}}(\tilde\varepsilon + \tilde\nu-\tilde\alpha)$,
\item\label{eq:stab-c} $\tilde\delta = \dom_{\fS_{n-1}}(\tilde\kappa + \tilde\varepsilon-\tilde\beta)$, and
\item\label{eq:stab-d} $\tilde\mu = \dom_{\fS_{n-1}}(\tilde\delta + \tilde\gamma-\tilde\varepsilon)$.
\end{enumerate}

Let us first reduce to the case where at least one of the staircases involved does not contain a zero.  Consider a square of staircases
\begin{equation*}
  \tikzset{ short/.style={ shorten >=7pt, shorten <=7pt } }
  \begin{tikzpicture}[xscale=4,yscale=1,baseline=(current bounding box.center)]
    \node (al) at (0,0) {\makebox[3cm]{$\alpha$}};
    \node (be) at (0,1) {\makebox[3cm]{$\beta=\alpha\pm e_i$}};
    \node (ga) at (1,0) {\makebox[3cm]{$\gamma=\dom_{\fS_n}(\alpha\pm e_j)$}};
    \node (de) at (1,1) {\makebox[3cm]{$\delta=\alpha\pm e_i\pm e_j$}};
    \draw[->] (al) -- (be) node [midway, left] {};
    \draw[->] (ga) -- (de) node [midway, right] {};
    \draw[->] (al) -- (ga) node [midway, below] {};
    \draw[->] (be) -- (de) node [midway, above] {};
  \end{tikzpicture}
\end{equation*}
where all of $\alpha$, $\beta$, $\gamma$ and $\delta$ contain a zero.  We first show that there is an index $k\not\in\{i,j\}$ such that $\alpha_k=\beta_k=\delta_k=0$.  Suppose on the contrary that $\alpha_k\neq 0$ for all $k\not\in\{i,j\}$.  Then, since $\beta$ contains a zero, we have $i\neq j$.  Furthermore, we have
\[
\begin{array}{lcll}
  \alpha_i=0&\text{or}&\alpha_j=0,\quad\text{and}\\
  \alpha_i=\mp1&\text{or}&\alpha_j=0,\quad\text{and}\\
  \alpha_i=0&\text{or}&\alpha_j=\mp1,\quad\text{and}\\
  \alpha_i=\mp1&\text{or}&\alpha_j=\mp1
\end{array}
\]
because $\alpha$, $\beta$, $\gamma$ and $\delta$ contain a zero, respectively.  However, this set of equations admits no solution.  Thus, there must be a further zero in $\alpha$ and therefore also in $\beta$, $\gamma$ and $\delta$.  From this it follows that $\tilde\gamma=\dom_{\fS_{n-1}}(\tilde\alpha + \tilde\delta-\tilde\beta)$.

Returning to the square in \eqref{eq:square-of-staircases}, we show
that $\varepsilon$ contains a zero entry if $\beta$ or $\gamma$
do. Suppose on the contrary that $\varepsilon$ does not contain a
zero entry.  Then $\varepsilon=\beta+e_i$, where $i$ is the position
of the (only) zero in $\beta$.  Moreover, we have $\alpha=\beta$,
because there is only one way to obtain a zero entry in $\alpha$ by
subtracting a unit vector.  Thus,
\[
\lambda=\dom_{\fS_n}(\beta+\alpha-\varepsilon) = \dom_{\fS_n}(\beta-e_i) = \beta-e_i,
\]
which implies that $\lambda$ does not contain a zero entry, contradicting our assumption.  Similarly, if $\gamma$ contains a zero at position $i$, we have $\varepsilon=\gamma+e_i$, $\gamma=\delta$ and $\mu=\dom_{\fS_n}(\gamma-e_i)$, a contradiction.

There remain three different cases:

\noindent{\bf $\beta$ contains a zero, but $\gamma$ does not.}

We have to show Equations~\eqref{eq:stab-b} and~\eqref{eq:stab-d}.
Let $\alpha=\varepsilon-e_i$ and $\nu=\varepsilon-e_i-e_j$.  Then
$\gamma=\dom_{\fS_n}(\varepsilon-e_j)$.  Since, by the foregoing,
$\varepsilon$ contains a zero, we have $\varepsilon_j=0$.  Since
$\alpha$ also has a zero we have $i\neq j$.  Since $\nu$ has a zero,
$\varepsilon_i=1$. Because $\gamma$ has no zero, $\mu=\nu$.  Together
with the fact that $\delta$ has a zero, this implies that
$\delta=\varepsilon-e_i$.  The equations can now be checked directly.

\noindent{\bf $\beta$ contains no zero, but $\gamma$ does.}

We have to show Equations~\eqref{eq:stab-a} and~\eqref{eq:stab-c}. Let $\lambda=\beta-e_i$ and $\alpha=\beta-e_i+e_j$.  Then $\varepsilon=\dom_{\fS_n}(\beta+e_j)$.  Since $\beta$ has no zero, but, by the foregoing, $\varepsilon$ does, we have $\beta_j=\bar 1$. Since $\lambda$ has a zero, $\beta_i=1$, and thus $i\neq j$. Because $\beta$ has no zero, $\kappa=\lambda$.  Again, the equations can now be checked directly.

\noindent{\bf None of $\beta$, $\epsilon$ and $\gamma$ contain a zero.}

In this case, $\kappa=\lambda$, $\delta=\alpha$ and $\mu=\nu$.  Let
$\lambda=\beta-e_i$, $\alpha=\beta-e_i+e_j$.  Then
$\varepsilon=\dom_{\fS_n}(\beta+e_j)$.  Thus $\beta_j\neq\bar 1$,
$\beta_i=1$, $\beta_{i+1}\leq\bar 1$ and, because $\alpha\neq\beta$,
we have $i\neq j$.  Because $\alpha$ and $\beta$ are staircases,
$\beta+e_j$ has in fact decreasing entries and
$\varepsilon=\beta+e_j$.  Thus, $\alpha=\varepsilon-e_i$,
$\nu=\varepsilon-e_i-e_k$ and $\gamma=\dom_{\fS_n}(\varepsilon-e_k)$.
Again, because $\varepsilon$ and $\nu$ are staircases,
$\varepsilon-e_k$ has decreasing entries and
$\gamma=\varepsilon-e_k$.  Thus, the equations can now be checked
directly.
\end{proof}

\subsection{Growth diagrams for staircase tableaux}

In this section we set up the notation used in the remaining
sections.  In particular, we slightly modify and generalise the
definition of $\G(\alt)$ from Section~\ref{sec:Stembridge}.
\begin{defn}
  For a pair of partitions $\wgt=(\pos\wgt,\neg\wgt)$, the partition
  $\pos\wgt$ is the \Dfn{positive part} and the partition $\neg\wgt$
  is the \Dfn{negative part}.  Given an integer $n$ not smaller than
  the sum of the lengths of the two partitions,
  $[\pos\wgt,\neg\wgt]_n$ is the staircase
  \[
  ({\pos\wgt}{}_{,0},\pos\wgt{}_{,1},\dots,0,\dots,0,
  \dots,-\neg\wgt{}_{,1},-\neg\wgt{}_{,0}).
  \]

  A \Dfn{staircase tableau} is a sequence of staircases
  $\alt=(\wgt^0,\wgt^1,\dots,\wgt^r)$ such that $\wgt^i$ and
  $\wgt^{i+1}$ differ by a unit vector for $0\leq i<r$.  If
  $\wgt^0=\emptyset$ the tableau is \Dfn{straight}, otherwise it is
  \Dfn{skew}. Unless explicitly stated, we consider only straight
  staircase tableaux.

  The \Dfn{extent}\footnote{It might be more logical to use \lq
    height\rq\ for the extent of a staircase, and \lq length\rq\ for
    the number $n$. However, Stembridge defines the height of a
    staircase as the number $n$. We therefore avoid the words \lq
    length\rq\ and \lq height\rq\ in the context of staircases
    altogether.} $\extent(\wgt)$ of a staircase
  $\wgt=[\pos\wgt,\neg\wgt]_n$ is the number of nonzero entries in
  $\wgt$.  Put differently, the extent is the sum of the lengths of
  the partitions $\pos\wgt$ and $\neg\wgt$.  The \Dfn{extent} of a
  staircase tableau is the maximal extent of its staircases.
\end{defn}

Given a staircase tableau we can create a growth diagram similar to
the procedure used in Section~\ref{sec:Stembridge}.  However, it will
be convenient to label \emph{all} corners of the cells with
staircases, instead of labelling the corners which are not on the
main diagonal or first superdiagonal with a partition instead of a
staircase.
\begin{defn}
  The growth diagram \Dfn{$\G(\alt)$} corresponding to a (straight)
  staircase tableau $\alt$ is obtained in analogy to
  Definition~\ref{defn:P}: label the top-left corner with the
  staircase $\wgt^0$.  If $\wgt^{i+1}$ is obtained from $\wgt^i$ by
  adding (respectively subtracting) a unit vector, $\wgt^{i+1}$
  labels the corner to the right of (respectively below) the corner
  labelled $\wgt^i$.  \emph{All} the remaining corners of $\G(\alt)$
  are then labelled with staircases as follows.  The positive parts
  on the corners to the left and below the path defined by the
  staircase tableau are obtained by applying the backward rule,
  whereas the forward rule determines the positive parts on the
  remaining corners. The negative parts are computed similarly.
\end{defn}
  
Alternatively, we can also create a growth diagram given a partial
filling and two partial standard Young tableaux.
\begin{defn}
  A \Dfn{partial filling} $\phi$ is a rectangular array of cells,
  where every row and every column contains at most one cell with a
  cross.

  Let $\phi$ be a partial filling having crosses in all rows except
  $\mathcal R$ (counted from the top), and in all columns except
  $\mathcal C$ (counted from the left).  Let $P$ and $Q$ be partial
  standard Young tableaux having entries $\mathcal R$ and
  $\mathcal C$ respectively.  Then the growth diagram
  \Dfn{$\G(\phi, P, Q)$} is obtained as follows.  The sequence of
  partitions corresponding to $Q$ (respectively $P$) determines the
  positive (respectively negative) parts of the staircases on the
  bottom (respectively right) border.  The remaining positive and
  negative parts are computed using the forward rule.

  If $\phi$ contains precisely one cross in every row and every
  column, we abbreviate $\G(\phi,\emptyset,\emptyset)$ to
  \Dfn{$\G(\phi)$}.
\end{defn}

Finally, any growth diagram in the sense above can be decomposed into
two classical growth diagrams, where all corners are labelled by
partitions.
\begin{defn}
  \Dfn{$\pos\G$} (respectively \Dfn{$\neg\G$}) denotes the
  (classical) growth diagrams obtained by ignoring the negative
  (respectively positive) parts of the staircases labelling the
  corners of a growth diagram $\G$.
\end{defn}
\begin{rmk}
  The classical growth diagram associated to a (partial) filling $\phi$ is precisely $\pos\G(\phi, Q)$.
\end{rmk}
\begin{rmk}
  Two horizontally adjacent shapes in $\pos\G(\phi, Q)$ differ if and only if there is no cross above in this column.  Two horizontally adjacent shapes in $\neg\G(\phi, Q)$ differ if and only if there is a cross above in this column.
\end{rmk}
\begin{rmk}\label{prop:transposed-growth}
  Transposing a filling $\phi$ is equivalent to interchanging $\pos\G(\phi)$ and $\neg\G(\phi)$.
\end{rmk}

Finally, we introduce the operations on fillings we want to relate to promotion.
\begin{defn}
  Let $\phi$ be a filling of a square grid.  The \Dfn{column rotation $\crot\phi$} (respectively, \Dfn{row rotation $\rrot\phi$}) of the filling $\phi$ is obtained from $\phi$ by removing the first column (respectively, row) and appending it at the right (respectively, bottom).

  The \Dfn{rotation $\rot\phi$} of a filling $\phi$ is $\crot\rrot\phi$.
\end{defn}

\subsection{Promotion and evacuation of alternating tableaux}
\label{sec:alternating}
In this section we prove Theorems~\ref{thm:s1q-partial-permutations}
and~\ref{thm:rotation-promotion-permutations}, with the exception of
the case $n=2$.

Let us first recall a classical fact concerning the effect of
removing the first column of a filling on the growth diagram in terms
of Schützenberger's jeu de taquin.
\begin{prop}[\protect{\cite[A~1.2.10]{MR1676282}}]
  \label{prop:jdt-delete-column}
  Consider the classical growth diagrams $\G$ and $\half\G$ for the
  partial fillings $\phi$ and $\half\phi$, where $\half\phi$ is
  obtained from $\phi$ by deleting its first column. Let $Q$ and
  $\half Q$
  % (respectively $P$ and $\half P$)
  be the partial standard Young tableaux corresponding to the
  sequence of partitions on the top borders
  % (respectively right borders)
  of the growth diagrams $\G$ and $\half\G$.  Then
  $\half Q = \jdt Q$, the tableau obtained by applying
  Schützenberger's jeu de taquin to $Q$.
%  and
%    \[
%    P = \begin{cases}
%    k\downarrow \half P&\text{if there is a cross in row $k$ of the first column of $\phi$}\\
%    \half P&\text{otherwise}.
%    \end{cases}
%    \]
\end{prop}

The following central result connects the local rule for the symmetric group with column rotation, the operation of moving the first letter of a permutation to the end.
\begin{thm}\label{thm:growth-diagram-rotation}
  Let $\phi$ be a filling of an $r\times r$ square grid having
  exactly one cross in every row and in every column. Let $\lambda$
  and $\nu$ be two adjacent staircases in $\G(\phi)$, not on the left
  border of $\G(\phi)$, and $\lambda$ being to the left of $\nu$ or
  above $\nu$.  Finally, let $\kappa$ and $\mu$ be the two
  corresponding staircases in $\G(\crot\phi)$, that is, the column
  index of $\kappa$ in $\G(\crot\phi)$ is one less than the column
  index of $\lambda$ in $\G(\phi)$.  Then, provided that
  $n\geq\max(\extent(\kappa),\extent(\lambda),\extent(\mu),\extent(\nu))$,
  we have $\mu = \dom_{\fS_n}(\kappa+\nu-\lambda)$.
\end{thm}

Because the filling and the staircases of a growth diagram determine
each other uniquely, we immediately obtain the following corollary.
\begin{cor}\label{cor:growth-diagram-rotation}
  Let $\phi$ be a filling of an $r\times r$ square grid having
  exactly one cross in every row and in every column.  Suppose that
  the staircases in $\alt=(1=\wgt^1,\dots,\wgt^{2r}=\emptyset)$ label
  a sequence of adjacent corners from the corner just to the right of
  the top-left corner to the bottom-right corner of $\G(\phi)$.
  Furthermore, suppose that the staircases
  $\half\alt=(\emptyset=\hm^0,\dots,\hm^{2r}=\emptyset)$ satisfy
  $\hm^i = \dom_{\fS_n}(\hm^{i-1}+\wgt^{i+1}-\wgt^i)$ for
  $i\leq 2r-1$.  Then, provided that
  $n\geq\max(\extent(\alt),\extent(\half\alt))$, the filling of
  $\G(\half\alt)$ is $\crot\phi$.
%
%  The analogous statement for the row rotation $\phi^{\rrot}$ of $\phi$ is obtained by transposing the filling and applying Proposition~\ref{prop:transposed-growth}.
\end{cor}
We remark that
% the first part of
Proposition~\ref{prop:jdt-delete-column}, restricted to permutations,
is a special case of Theorem~\ref{thm:growth-diagram-rotation}.  More
precisely, it is obtained by considering the staircase tableau
$(1=\wgt^1,\dots,\wgt^{2r}=\emptyset)$ consisting of the partitions
labelling the corners along the top and the right border of a
classical growth diagram, with the empty shape in the top-left corner
removed.

It is not hard to extend the theorem to partial fillings; the
statement is completely analogous.  Its proof proceeds by extending
the partial filling to a permutation.  However, it turns out to be
more convenient to deduce the statements for staircase tableaux of
non-empty shape from the corresponding statements for staircase
tableaux of empty shape directly.

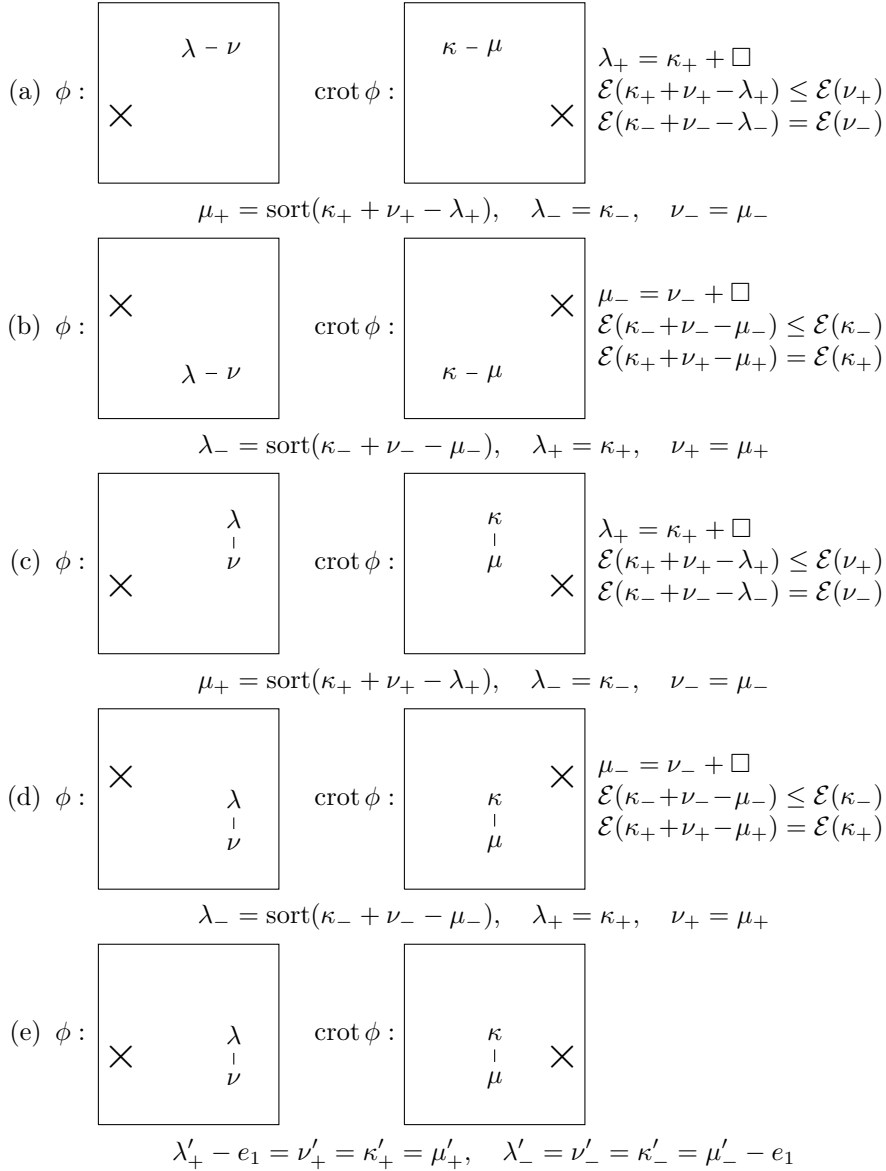
\begin{figure}
    \begin{enumerate}[(a)]
    \item
        $\phi:$ \begin{tikzpicture}[baseline={([yshift=-.5ex]current bounding box.center)},scale=0.6]
        \node (la) at (2,3) {$\lambda$};
        \node (nu) at (3,3) {$\nu$};
        \draw[-] (la) -- (nu);
        \node at (0.5,1.5) {\huge$\times$};
        \draw[-] (0,0) -- (0,4) -- (4,4) -- (4,0) -- (0,0);
      \end{tikzpicture} \quad $\crot\phi:$
        \begin{tikzpicture}[baseline={([yshift=-.5ex]current bounding box.center)},scale=0.6]
        \node (la) at (1,3) {$\kappa$};
        \node (nu) at (2,3) {$\mu$};
        \draw[-] (la) -- (nu);
        \node at (3.5,1.5) {\huge$\times$};
        \draw[-] (0,0) -- (0,4) -- (4,4) -- (4,0) -- (0,0);
      \end{tikzpicture}
      \begin{minipage}{3.8cm}
      $\pos\lambda=\pos\kappa+\square$\\
      $\extent(\pos\kappa+\pos\nu-\pos\lambda)\leq\extent(\pos\nu)$\\
      $\extent(\neg\kappa+\neg\nu-\neg\lambda)=\extent(\neg\nu)$
      \end{minipage}
      \[\pos\mu = \domS(\pos\kappa+\pos\nu-\pos\lambda),\quad  \neg\lambda = \neg\kappa,\quad \neg\nu = \neg\mu\]

      \item $\phi:$ \begin{tikzpicture}[baseline={([yshift=-.5ex]current bounding box.center)},scale=0.6]
        \node (la) at (2,1) {$\lambda$};
        \node (nu) at (3,1) {$\nu$};
        \draw[-] (la) -- (nu);
        \node at (0.5,2.5) {\huge$\times$};
        \draw[-] (0,0) -- (0,4) -- (4,4) -- (4,0) -- (0,0);
      \end{tikzpicture} \quad $\crot\phi:$
        \begin{tikzpicture}[baseline={([yshift=-.5ex]current bounding box.center)},scale=0.6]
        \node (la) at (1,1) {$\kappa$};
        \node (nu) at (2,1) {$\mu$};
        \draw[-] (la) -- (nu);
        \node at (3.5,2.5) {\huge$\times$};
        \draw[-] (0,0) -- (0,4) -- (4,4) -- (4,0) -- (0,0);
      \end{tikzpicture}
      \begin{minipage}{3.8cm}
      $\neg\mu=\neg\nu+\square$\\
      $\extent(\neg\kappa+\neg\nu-\neg\mu)\leq\extent(\neg\kappa)$\\
      $\extent(\pos\kappa+\pos\nu-\pos\mu)=\extent(\pos\kappa)$
      \end{minipage}
      \[\neg\lambda = \domS(\neg\kappa+\neg\nu-\neg\mu),\quad  \pos\lambda = \pos\kappa,\quad \pos\nu = \pos\mu\]

      \item $\phi:$ \begin{tikzpicture}[baseline={([yshift=-.5ex]current bounding box.center)},scale=0.6]
        \node (la) at (3,3) {$\lambda$};
        \node (nu) at (3,2) {$\nu$};
        \draw[-] (la) -- (nu);
        \node at (0.5,1.5) {\huge$\times$};
        \draw[-] (0,0) -- (0,4) -- (4,4) -- (4,0) -- (0,0);
      \end{tikzpicture} \quad $\crot\phi:$
        \begin{tikzpicture}[baseline={([yshift=-.5ex]current bounding box.center)},scale=0.6]
        \node (la) at (2,3) {$\kappa$};
        \node (nu) at (2,2) {$\mu$};
        \draw[-] (la) -- (nu);
        \node at (3.5,1.5) {\huge$\times$};
        \draw[-] (0,0) -- (0,4) -- (4,4) -- (4,0) -- (0,0);
      \end{tikzpicture}
      \begin{minipage}{3.8cm}
      $\pos\lambda=\pos\kappa+\square$\\
      $\extent(\pos\kappa+\pos\nu-\pos\lambda)\leq\extent(\pos\nu)$\\
      $\extent(\neg\kappa+\neg\nu-\neg\lambda)=\extent(\neg\nu)$
      \end{minipage}
      \[\pos\mu = \domS(\pos\kappa+\pos\nu-\pos\lambda),\quad  \neg\lambda = \neg\kappa,\quad \neg\nu = \neg\mu\]

      \item $\phi:$ \begin{tikzpicture}[baseline={([yshift=-.5ex]current bounding box.center)},scale=0.6]
        \node (la) at (3,2) {$\lambda$};
        \node (nu) at (3,1) {$\nu$};
        \draw[-] (la) -- (nu);
        \node at (0.5,2.5) {\huge$\times$};
        \draw[-] (0,0) -- (0,4) -- (4,4) -- (4,0) -- (0,0);
      \end{tikzpicture} \quad $\crot\phi:$
        \begin{tikzpicture}[baseline={([yshift=-.5ex]current bounding box.center)},scale=0.6]
        \node (la) at (2,2) {$\kappa$};
        \node (nu) at (2,1) {$\mu$};
        \draw[-] (la) -- (nu);
        \node at (3.5,2.5) {\huge$\times$};
        \draw[-] (0,0) -- (0,4) -- (4,4) -- (4,0) -- (0,0);
      \end{tikzpicture}
      \begin{minipage}{3.8cm}
      $\neg\mu=\neg\nu+\square$\\
      $\extent(\neg\kappa+\neg\nu-\neg\mu)\leq\extent(\neg\kappa)$\\
      $\extent(\pos\kappa+\pos\nu-\pos\mu)=\extent(\pos\kappa)$
      \end{minipage}
      \[\neg\lambda = \domS(\neg\kappa+\neg\nu-\neg\mu),\quad  \pos\lambda = \pos\kappa,\quad \pos\nu = \pos\mu\]

      \item $\phi:$ \begin{tikzpicture}[baseline={([yshift=-.5ex]current bounding box.center)},scale=0.6]
        \node (la) at (3,2) {$\lambda$};
        \node (nu) at (3,1) {$\nu$};
        \draw[-] (la) -- (nu);
        \node at (0.5,1.5) {\huge$\times$};
        \draw[-] (0,0) -- (0,4) -- (4,4) -- (4,0) -- (0,0);
      \end{tikzpicture} \quad $\crot\phi:$
        \begin{tikzpicture}[baseline={([yshift=-.5ex]current bounding box.center)},scale=0.6]
        \node (la) at (2,2) {$\kappa$};
        \node (nu) at (2,1) {$\mu$};
        \draw[-] (la) -- (nu);
        \node at (3.5,1.5) {\huge$\times$};
        \draw[-] (0,0) -- (0,4) -- (4,4) -- (4,0) -- (0,0);
      \end{tikzpicture}
      \[ \pos{\lambda'} - e_1 = \pos{\nu'} = \pos{\kappa'} = \pos{\mu'}, \quad  \neg{\lambda'} = \neg{\nu'}= \neg{\kappa'} = \neg{\mu'} - e_1\]
    \end{enumerate}
    \caption{The cases considered in the proof of
      Theorem~\ref{thm:growth-diagram-rotation}.}
    \label{fig:different_cases_column_rotation}
\end{figure}
\begin{proof}[Proof of Theorem~\ref{thm:growth-diagram-rotation}]
  {\bf Local rules for the positive and the negative parts.}  

  Let us first determine certain local rules satisfied separately by
  the positive and negative parts of the staircases $\kappa$,
  $\lambda$, $\mu$ and $\nu$. A summary of the various cases is
  displayed in Figure~\ref{fig:different_cases_column_rotation},
  where the rules we verify are displayed below the corresponding
  diagrams.  In the following, addition and subtraction of integer
  partitions is defined by interpreting them as vectors in $\ZZ^n$.

  {\em First case, $\lambda$ left of $\nu$:} 

  Let
  $Q=(\emptyset=\mu_0, \mu_1, \ldots, \mu_{s-1} = \pos\lambda, \mu_s
  = \pos\nu)$
  be the partial standard Young tableau corresponding to the sequence
  of partitions in $\pos\G(\phi)$ on the same row as $\lambda$ and
  $\nu$, beginning at the left border. Let
  $\half Q = (\emptyset=\half\mu_0, \half\mu_1, \ldots,
  \half\mu_{s-2} = \pos\kappa, \half\mu_{s-1} = \pos\mu)$
  be the corresponding partial standard Young tableau in
  $\pos\G(\crot\phi)$.

  Suppose there is a cross in $\phi$ in the first column in a row
  below $\nu$, as in
  Figure~\ref{fig:different_cases_column_rotation}.a. Then, by
  Proposition~\ref{prop:jdt-delete-column}, $\half Q = \jdt Q$.  This
  implies that the partitions
  $\mu_{s-1},\mu_s,\half\mu_{s-2},\half\mu_{s-1}$ satisfy the local
  (growth diagram) rule
  $\half\mu_{s-1} = \domS(\half\mu_{s-2}+\mu_s-\mu_{s-1})$, that is,
  $\pos\mu = \domS(\pos\kappa+\pos\nu-\pos\lambda)$.  Moreover
  $\neg\kappa = \neg\lambda$ and $\neg\nu=\neg\mu$ because the growth
  for the negative parts of the staircases is from the top-right to
  the bottom-left.

  If there is a cross in the first column in a row above $\nu$, as in
  Figure~\ref{fig:different_cases_column_rotation}.b, we reason in a
  very similar way.  In this case $\pos\kappa = \pos\lambda$ and
  $\pos\nu=\pos\mu$. For the negative parts of the staircases we
  consider the partial standard Young tableaux $Q$ and $\half Q$
  corresponding to the sequences of partitions beginning at the
  \emph{right} border of $\neg\G(\phi)$ and $\neg\G(\crot\phi)$
  respectively. We then have $Q = \jdt \half Q$ and conclude
  $\neg\lambda = \domS(\neg\kappa+\neg\nu-\neg\mu)$ as before, using
  the symmetry of the local rule, see
  Remark~\ref{rmk:symmetry-of-local-rule}.

  {\em Second case, $\lambda$ above $\nu$:}

  Depending on the position of the cross in the first column there
  are three slightly different cases, as illustrated in
  Figure~\ref{fig:different_cases_column_rotation}.c,~d and~e.

  Recall that the partitions on the right border of a (classical)
  growth diagram corresponding to the right to left reversal of a
  filling $\psi$ are obtained by transposing the partitions on the
  right border of the (classical) growth diagram corresponding to
  $\psi$.  Consider now the filling $\psi$ below and to the left of
  $\lambda$, and let $\half\psi$ be the filling to the left and below
  $\kappa$.  Note that the reversal of $\psi$ is obtained from the
  reversal of $\half\psi$ by appending the first column of $\psi$ to
  the right.  We thus obtain that the transposes of the positive
  parts of the staircases $\kappa$, $\mu$, $\lambda$ and $\nu$
  satisfy the local (growth diagram) rule.

  The relation between the negative parts of the staircases $\kappa$,
  $\mu$, $\lambda$ and $\nu$ is obtained in a very similar way by
  considering the fillings above and to the right of $\nu$ and $\mu$.

  {\bf Bounding the extent and deducing the local rule.}  We now show
  $\mu = \dom_{\fS_n}(\kappa+\nu-\lambda)$, provided
  $n\geq\max(\extent(\kappa),\extent(\lambda),\extent(\mu),\extent(\nu))$.
  To do so, we extend the notion of extent to arbitrary vectors with
  all entries non-negative or all entries non-positive: for a vector
  $\pos\alpha\in\ZZ_{\geq0}^n$, the extent $\extent(\pos\alpha)$ is
  $n$ minus the number of trailing zeros.  Similarly, for a vector
  $\neg\alpha\in\ZZ_{\leq0}^n$, the extent $\extent(\neg\alpha)$ is
  $n$ minus the number of leading zeros.

  The case illustrated in
  Figure~\ref{fig:different_cases_column_rotation}.e follows by
  direct inspection.  We thus only consider the remaining four cases.
  Because $\pos\mu$ and $\neg\mu$ are obtained by sorting
  $\pos\kappa+\pos\nu-\pos\lambda$ and
  $\neg\kappa+\neg\nu-\neg\lambda$ respectively, the latter must have
  all entries non-negative.  Similarly, also
  $\pos\kappa+\pos\nu-\pos\mu$ and $\neg\kappa+\neg\nu-\neg\mu$ have
  all entries non-negative, because $\pos\lambda$ and $\neg\lambda$
  are obtained by sorting these vectors, by the symmetry of the local
  rule, see Remark~\ref{rmk:symmetry-of-local-rule}.

  Suppose first that
  $n\geq\max(\extent(\kappa),\extent(\lambda),\extent(\nu))$. Then
  \[
  \dom_{\fS_n}(\kappa+\nu-\lambda)=\dom_{\fS_n}([\pos\kappa,\neg\kappa]_n+[\pos\nu,\neg\nu]_n-[\pos\lambda,\neg\lambda]_n)
  =\dom_{\fS_n}(\pos\alpha + \neg\alpha),
  \]
  where
  $\pos\alpha =
  [\pos\kappa,\emptyset]_n+[\pos\nu,\emptyset]_n-[\pos\lambda,\emptyset]_n$
  and
  $\neg\alpha=[\emptyset,\neg\kappa]_n+[\emptyset,\neg\nu]_n-[\emptyset,\neg\lambda]_n$.
  It remains to show that
  \[
  \extent(\pos\alpha) + \extent(\neg\alpha)
  =\extent\big(\pos\kappa+\pos\nu-\pos\lambda\big) +
  \extent\big(\neg\kappa+\neg\nu-\neg\lambda\big)\leq n,
  \]
  because then
  \begin{multline*}
    \dom_{\fS_n}(\pos\alpha+\neg\alpha) = [\domS(\pos\alpha),\domS(\neg\alpha)]_n\\
    = [\domS(\pos\kappa+\pos\nu-\pos\lambda),
    \domS(\neg\kappa+\neg\nu-\neg\lambda)]_n = [\pos\mu,\neg\mu]_n =
    \mu.
  \end{multline*}
  Similarly, suppose that
  $n\geq\max(\extent(\kappa),\extent(\mu),\extent(\nu))$. In this
  case, reasoning as above, we have to show that
  $\extent\big(\pos\kappa+\pos\nu-\pos\mu\big) +
  \extent\big(\neg\kappa+\neg\nu-\neg\mu\big)\leq n$.

  The first inequality is verified by inspection of
  Figure~\ref{fig:different_cases_column_rotation}.a and~c, whereas
  the second concerns
  Figure~\ref{fig:different_cases_column_rotation}.b and~d.  Here we
  write, for example, $\pos\lambda=\pos\kappa+\square$ to indicate
  that the partition $\pos\lambda$ is obtained from the partition
  $\pos\kappa$ by adding a single cell, which implies the inequality
  for the extent.
\end{proof}

\begin{defn}\label{def:half-promotion}
  Let $\alt=(\emptyset = \wgt^0,\wgt^1,\dots,\wgt^{2r-1},\wgt^{2r}=\mu)$
  be an alternating tableau.  Then \Dfn{$\halfpr\alt=(\emptyset = \hm^0,\hm^1,\dots,\hm^{2r-1},\hm^{2r}=\mu)$} is the staircase tableau obtained from $\alt$ by setting $\hm^1=\wgt^1=1$, and then applying the local rule~\eqref{eq:local} successively to $\wgt^i$, $\wgt^{i+1}$, and $\hm^{i-1}$ to obtain $\hm^i$ for $i\leq 2r-1$.  Additionally, we set $\hm^{2r}=\mu$.
\end{defn}
In other words, $\halfpr\alt$ can be read off from the diagram for promotion as illustrated in diagram~\eqref{eq:pr-adjoint-schema} beginning with the empty shape in the lower left corner, then following the second row, and terminating with the shape $\mu$ in the upper right corner.

\begin{lemma}\label{lem:bound}
  \begin{enumerate}[(a)]
  \item\label{it:staircase} Let $\alt$ be a staircase tableau of empty
    shape and length $r$.  Then the extent of $\alt$ is at most $r$.
  \item\label{it:alternating} Restricting to alternating tableaux, there is
    a single alternating tableau $\alt_0$ of empty shape, length $r$
    and extent $r$. The filling $\phi_0$ of its growth diagram
    $\G(\alt_0)$ is invariant under rotation: $\rot\phi_0=\phi_0$.
  \item\label{it:alternating-even} Restricting further to alternating
    tableaux of even length, the only tableau $\alt$ such that
    $\half\pr\alt$ has extent $r$ is $\alt_0$.
  \end{enumerate}
\end{lemma}
\begin{proof}
  Statement~\eqref{it:staircase} is trivial.  To see
  statement~\eqref{it:alternating}, note that the unique length~$r$
  alternating tableau of empty shape with maximal extent is, for
  $r=2s+1$ odd,
  \[
    \alt_0=(\emptyset,1,1\bar1,\dots,%
    1^s\bar1^s,1^{s+1}\bar1^s,1^s\bar1^s,\dots,%
    1\bar1,1,\emptyset)
  \]
  and, for $r=2s$ even,
  \[
    \alt_0=(\emptyset,1,1\bar1,\dots,%
    1^s\bar1^s,1^s\bar1^{s-1},1^{s-1}\bar1^{s-1},\dots,%
    1\bar1,1,\emptyset).
  \]
  In both cases the extent is $r$ and the corresponding permutation
  is, in one line notation, $s+1,s+2,\dots,r,1,\dots,s$.  This
  permutation is invariant under rotation.

  Similarly, to see statement~\eqref{it:alternating-even}, a
  staircase tableau $\half\pr\alt$ with extent $r$ must have filling
  corresponding to the permutation $s,s+1,\dots,r,1,\dots,s-1$, which
  is the filling $\rrot\phi$, and thus $\alt=\alt_0$.
\end{proof}

The first statement of Theorem~\ref{thm:rotation-promotion-permutations}, with the exception of the case $n=2$ and the case $n=r-1$, is a direct consequence of the following result.  The special case of $\GL(2)$-alternating tableaux and $\GL(r-1)$-alternating tableaux will be considered in Section~\ref{sec:GL2} below.
\begin{thm}\label{thm:pr-rot}
  Let $\alt$ be a $\GL(n)$-alternating tableau of length $r$ and
  empty shape.  Let $\phi$ be the filling of the growth diagram
  $\G(\alt)$.  Let $\half\phi$ and $\hat\phi$ be the fillings of the
  growth diagrams $\G(\halfpr\alt)$ respectively $\G(\pr\alt)$.

  Then, for $n\geq r$, $\rrot\phi=\half\phi$ and
  $\crot\half\phi = \hat\phi$.
\end{thm}
\begin{proof}
  Let
  \begin{align*}
    \halfpr\alt&=(\emptyset = \hm^0,\hm^1,\dots,\hm^{2r-1}=1,\hm^{2r}=\emptyset)\\
    \intertext{and let}
    \pr\alt &= (\emptyset = \Pm^0,\Pm^1,\dots,\Pm^{2r-1}=1,\Pm^{2r}=\emptyset).\\
    \intertext{Furthermore, let}
    \widetilde{\alt}&=(\emptyset = \widetilde{\mu}^0,\widetilde{\mu}^1,\dots,\widetilde{\mu}^{2r-1}=\bar1,\widetilde{\mu}^{2r}=\emptyset)
  \end{align*}
  be the staircase tableau obtained by setting $\widetilde{\mu}^0=\emptyset$ and then applying the local rule~\eqref{eq:local} successively to $\hm^i$, $\hm^{i+1}$, and $\widetilde{\mu}^{i-1}$ to obtain $\widetilde{\mu}^i$ for $i\leq 2r-1$.  Because of Lemma~\ref{lem:bound} and the assumption $n\geq r$, Corollary~\ref{cor:growth-diagram-rotation} is applicable and implies that the filling $\widetilde{\phi}$ of the growth diagram $\G(\widetilde{\alt})$ is $\crot\half\phi$.

  All staircases in $\widetilde{\alt}$ except
  $\widetilde{\mu}^{2r-1}$ coincide with those of $\pr\alt$.  Because
  $\hm^{2r-1}=1$, $\hm^{2r}=\emptyset$ and $\widetilde{\mu}^{2r-2}$
  is either $\emptyset$ or $1\bar1$, we have
  $\widetilde{\mu}^{2r-1}=\bar 1$.  However, since $\widetilde{\phi}$
  and $\hat\phi$ correspond to permutations and the first $r-1$
  columns of these fillings are the same, we conclude that
  $\widetilde{\phi}$ equals $\hat\phi$.

  Because of the symmetry of the local rules pointed out in
  Remark~\ref{rmk:symmetry-of-local-rule} and because
  $\half{\mu}^{2r-1}=1$, we can apply the same reasoning replacing
  $\halfpr\alt$ and $\pr\alt$ with the reversal of $\halfpr\alt$ and
  the reversal of $\alt$.  Clearly, the filling corresponding to the
  reversal of a tableau is obtained by flipping the original filling
  over the diagonal from the bottom-left to the top-right.  In the
  process, column rotation is replaced by row rotation, which implies
  that $\rrot(\phi)=\half\phi$.
\end{proof}

\begin{cor}\label{cor:pr-rot-bounds}
  In the setting of Theorem~\ref{thm:pr-rot}, if $n$ is odd it is
  sufficient to require $n\geq r-1$.
\end{cor}
\begin{proof}
    Let $\alt$ be a $\GL(n)$-alternating tableau of length $r=2s$, with $n\geq r$. Let $\phi$ be the filling of $\G(\alt)$. Then, combining Lemma~\ref{lem:bound} and Theorem~\ref{thm:pr-rot} we obtain
    \[
    \extent(\alt) = r \Leftrightarrow \extent(\half\pr\alt) = r \Leftrightarrow \extent(\pr\alt) = r.
    \]

    By contraposition,
    $\extent(\alt) < r \Leftrightarrow \extent(\half\pr\alt) < r
    \Leftrightarrow \extent(\pr\alt) < r$.
    Thus, the claim follows using the proof of
    Theorem~\ref{thm:pr-rot}, taking into account that
    Corollary~\ref{cor:growth-diagram-rotation} is now applicable
    even with $n\geq r-1$.
\end{proof}

Finally, we can conclude one part of Theorem~\ref{thm:rotation-promotion-permutations}.  Note that the case of odd $n$ is also covered by the previous corollary.
\begin{cor}
    Let $\alt$ be a $\GL(n)$-alternating tableau of length $r$ and empty shape.  Then, for $n\geq r-1$, $\rot\Perm(\alt)=\Perm(\pr\alt)$.
\end{cor}
\begin{proof}
    This is a consequence of Lemma~\ref{lem:bound} and Theorem~\ref{thm:stability}.
\end{proof}

We now introduce a different way to obtain the filling of $\G(\alt)$.
This construction will also shed some additional light on the
relationship between the local rule~\eqref{eq:local} and those in
Figure~\ref{fig:growth-cell}.

Consider the evacuation diagram for obtaining the evacuation as
illustrated in Figure~\ref{fig:ev-adjoint-example}.  We construct a
filling of the cells surrounded by three or four staircases using the
symbols {\negcross}, {\poscross} and {\fixcross} as follows:
\[
  \begin{tikzpicture}[baseline={([yshift=-.5ex]current bounding box.center)},scale=2]
    \node (ka) at (0,0) {$[\alpha,\tau]_n$};
    \node (la) at (0,1) {$[\alpha,\sigma]_n$};
    \node (nu) at (1,1) {$[\alpha,\tau]_n$};
    \node (mu) at (1,0) {$[\beta,\tau]_n$};
    \draw[<-] (ka) -- (la);
    \draw[<-] (mu) -- (nu);
    \draw[->] (la) -- (nu);
    \draw[->] (ka) -- (mu);
    \node at (0.5,0.5)[rectangle,minimum size=1.5cm,draw] {\negcross};
  \end{tikzpicture}
  \text{or}\quad
  \begin{tikzpicture}[baseline={([yshift=-.5ex]current bounding box.center)},scale=2]
    \node (ka) at (0,0) {$[\alpha,\tau]_n$};
    \node (la) at (0,1) {$[\beta,\tau]_n$};
    \node (nu) at (1,1) {$[\alpha,\tau]_n$};
    \node (mu) at (1,0) {$[\alpha,\sigma]_n$};
    \draw[<-] (ka) -- (la);
    \draw[<-] (mu) -- (nu);
    \draw[->] (la) -- (nu);
    \draw[->] (ka) -- (mu);
    \node at (0.5,0.5)[rectangle,minimum size=1.5cm,draw] {\poscross};
  \end{tikzpicture}
  \text{or}\quad
  \begin{tikzpicture}[baseline={([yshift=-.5ex]current bounding box.center)},scale=2]
    \node (ka) at (0,0) {};
    \node (la) at (0,1) {$[1,\emptyset]_n$};
    \node (nu) at (1,1) {$[\emptyset,\emptyset]_n$};
    \node (mu) at (1,0) {$[1,\emptyset]_n$};
    \draw[<-] (mu) -- (nu);
    \draw[->] (la) -- (nu);
    \node at (0.5,0.5)[rectangle,minimum size=1.5cm,draw] {\fixcross};
  \end{tikzpicture},
\]
where $\beta$ (respectively $\sigma$) is obtained from $\alpha$
(respectively $\tau$) by adding a cell to the first column of the
Ferrers diagram of the partition.  All other cells remain empty.

  The following lemma is the main building block in establishing the connection between the filling of $\G(\alt)$ and the decorated evacuation diagram.
\begin{lemma}\label{lem:alternating-equalities}
  Let $\alt = (\emptyset=\wgt^0,\dots,\wgt^{2r}=\emptyset)$ be an
  alternating tableau of empty shape.  Let
  $\halfpr\alt=(\emptyset =
  \hm^0,\hm^1,\dots,\hm^{2r-1},\hm^{2r}=\emptyset)$ be as in
  Definition~\ref{def:half-promotion} and let
  $\pr\alt = (\emptyset = \Pm^0,\dots,\Pm^{2r}=\emptyset)$ be the
  promotion of $\alt$.  Suppose that the filling of $\G(\alt)$ has a
  cross (that is, a {\negcross}, {\poscross} or {\fixcross}) in
  column $\ell>1$ of the first row and in row $k>1$ of the first
  column.  Then, for even $n\geq r$ and for odd $n\geq r-1$, we have
\begin{enumerate}[(a)]
\item $\pos{\wgt^j} = \pos{\hm^{j-1}}$ for $2\leq j\leq 2\ell-2$,
\item $\neg{\wgt^j} = \neg{\hm^{j-1}}$ for $j> 2\ell-2$,
\item $\pos{\wgt^{2\ell-2}} = \pos{\wgt^{2\ell-1}} = \pos{\hm^{2\ell-3}}$, and $\pos{\hm^{2\ell-2}}$ is obtained from these by adding a cell to the first column.  The cell labelled with these four staircases contains a~{\negcross}.
\end{enumerate}
% If there is no cross in the first row, the first equality holds for $2\leq j\leq 2r$.

Similarly,
\begin{enumerate}[(a')]
\item $\neg{\hm}^j = \neg{\Pm^{j-1}}$ for $1\leq j\leq 2k-2$
\item $\pos{\hm}^j = \pos{\Pm^{j-1}}$ for $j > 2k-2$
\item $\pos{\hm^{2k-1}} = \pos{\Pm^{2k-2}} = \pos{\Pm^{2k-3}}$, and $\pos{\hm^{2k-2}}$ is obtained from these by adding a cell to the first column. The cell labelled with these four staircases contains a~{\poscross}.
\end{enumerate}
% If there is no cross in the first column, the first equality holds for $1\leq j\leq 2r-1$.

Finally, suppose that there is a cross in the top-left cell, that is
$k=\ell=1$. Then
\begin{enumerate}[(f)]
    \item $\wgt^1 = 1, \wgt^2 = \emptyset$ and $\hm^1=1$. The cell labelled with these four staircases contains a~{\fixcross}.
    \item[(f')] $\wgt^j=\hm^{j-1}-e_1=\Pm^{j-2}$ for all $2\le j \le 2r$.
\end{enumerate}
\end{lemma}
\begin{proof}
  Consider a square of four adjacent staircases in the diagram for
  computing the promotion of an alternating tableau below:
  \begin{equation}\label{eq:pr-adjoint}
    \def\arraystretch{2.5}
    \begin{array}{rcccccccl}
      \wgt^2&\dots&\tikzmark{s1}\wgt^{2\ell-2}&\wgt^{2\ell-1}             &\multicolumn{4}{c}{\dotfill}                  &\wgt^{2r}=\emptyset\\
      \hm^1&\dots&\hm^{2\ell-3}             &\hm^{2\ell-2}\tikzmark{e1}&\dots&\tikzmark{s2}\hm^{2k-2}&\hm^{2k-1}&\dots&\hm^{2r-1}\\
      \emptyset=\Pm^0&\multicolumn{4}{c}{\dotfill}                                     &\Pm^{2k-3}&\Pm^{2k-2}\tikzmark{e2}&\dots&\Pm^{2r-2}
    \end{array}
    \begin{tikzpicture}[remember picture,overlay]
      % \node[draw,line width=0pt,rectangle,minimum size=2cm,fit={(pic cs:s1) (pic cs:e1)},yshift=1pt] {};
      % \node[draw,line width=0pt,rectangle,minimum size=2cm,fit={(pic cs:s2) (pic cs:e2)},yshift=1pt] {};
      \node [xshift=28pt,yshift=-12pt]  at (pic cs:s1) {\negcross};
      \node [xshift=28pt,yshift=-12pt]  at (pic cs:s2) {\poscross};
    \end{tikzpicture}
    \begin{tikzpicture}[remember picture,overlay]
    \end{tikzpicture}
  \end{equation}
  By definition, these satisfy the local rule, as required by
  Theorem~\ref{thm:rotation-promotion-matchings}. By
  Corollary~\ref{cor:pr-rot-bounds},
  Theorem~\ref{thm:growth-diagram-rotation} is applicable with the
  given bounds for $n$.  The equalities for the staircases in the
  second and third row are precisely the equalities listed below the
  illustrations in Figure~\ref{fig:different_cases_column_rotation}:
  cases~(a) and~(c) there describe the situation to the left of
  \poscross, case~(e) describes the situation at \poscross\ and
  cases~(b) and~(d) describe the situation to the right of \poscross.

  The equalities for the staircases in the first and second row can
  be obtained as in the last paragraph of the proof of
  Theorem~\ref{thm:pr-rot}.
\end{proof}

By successively applying Lemma~\ref{lem:alternating-equalities}
we obtain the following result for the evacuation diagram.
\begin{cor}\label{cor:relate-fillings}
  Let $\alt$ be a $\GL(n)$-alternating tableau of empty shape and
  length $r$ with corresponding growth diagram $\G(\alt)$ and filling
  $\phi$. Suppose that $n\geq r$ if $n$ is even and $n\geq r-1$ if
  $n$ is odd. Consider the evacuation diagram with filling obtained
  as above.  A {\negcross} appears only in odd columns and odd rows,
  a {\poscross} appears only in even columns and even rows and a
  {\fixcross} appears only in even columns and odd rows. Moreover
  $(i,j)$ is the position of a cell with a cross in $\phi$ if and
  only if one of the following cases holds.
  \begin{itemize}
  \item $i<j$ and there is a {\negcross} in row $2i-1$ and column $2j-1$ in the evacuation diagram.
  \item $i>j$ and there is a {\poscross} in row $2j$ and column $2i$.
  \item There is a {\fixcross} in row $2i-1$ and column $2j$. Then we
    also obtain $i=j$.
  \end{itemize}
\end{cor}

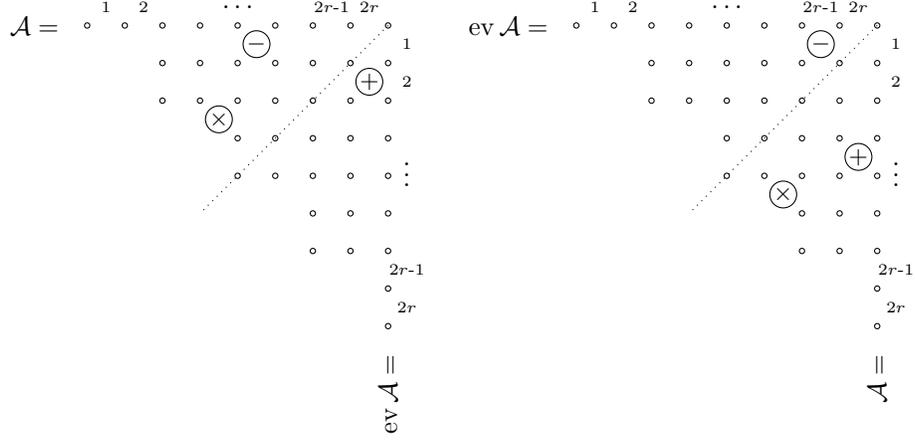
\begin{figure}[h]
  \begin{tikzpicture}[baseline=0pt,scale=0.5]
    \def\rr{8}
    % labels
    \node at (0.5,\rr.5) {\tiny$1$};
    \node at (1.5,\rr.5) {\tiny$2$};
    \node at (\rr/2,\rr.5) {$\ldots$};
    \node at (\rr-1.5,\rr.5) {\tiny$2r\text{-}1$};
    \node at (\rr-0.5,\rr.5) {\tiny$2r$};
    \node at (\rr.5,\rr-0.5) {\tiny$1$};
    \node at (\rr.5,\rr-1.5) {\tiny$2$};
    \node at (\rr.5,\rr/2+0.25) {$\vdots$};
    \node at (\rr.5,1.5) {\tiny$2r\text{-}1$};
    \node at (\rr.5,0.5) {\tiny$2r$};
    % circles for the partitions
    \foreach \y in {0,...,\rr} {
    	\tikzmath{\xmax = \y-Mod(\y,2);};
        \foreach \x in {0,...,\xmax}
        	\draw (\rr-\x,\y) circle[radius=2pt];
    }
    % axis
    \draw[dotted] (\rr,\rr) -- (\rr/2-1,\rr/2-1);
    % more labels
    \node[anchor=east] at (-0.5,\rr) {$\alt=$};
    \node[rotate=90,anchor=east] at (\rr,-0.5) {$\ev \alt=$};
    %the filling
    \node at (5-0.5,\rr-1+0.5) {\negcross};
    \node at (8-0.5,\rr-2+0.5) {\poscross};
    \node at (4-0.5,\rr-3+0.5) {\fixcross};
    \end{tikzpicture}\quad
    \begin{tikzpicture}[baseline=0pt,scale=0.5]
    \def\rr{8}
    %labels
    \node at (0.5,\rr.5) {\tiny$1$};
    \node at (1.5,\rr.5) {\tiny$2$};
    \node at (\rr/2,\rr.5) {$\ldots$};
    \node at (\rr-1.5,\rr.5) {\tiny$2r\text{-}1$};
    \node at (\rr-0.5,\rr.5) {\tiny$2r$};
    \node at (\rr.5,\rr-1+0.5) {\tiny$1$};
    \node at (\rr.5,\rr-2+0.5) {\tiny$2$};
    \node at (\rr.5,\rr/2+0.25) {$\vdots$};
    \node at (\rr.5,1.5) {\tiny$2r\text{-}1$};
    \node at (\rr.5,0.5) {\tiny$2r$};
    % circles for the partitions
    \foreach \y in {0,...,\rr} {
    	\tikzmath{\xmax = \y-Mod(\y,2);};
        \foreach \x in {0,...,\xmax}
        	\draw (\rr-\x,\y) circle[radius=2pt];
    }
    % axis
    \draw[dotted] (\rr,\rr) -- (\rr/2-1,\rr/2-1);
    % more labels
    \node[anchor=east] at (-0.5,\rr) {$\ev \alt=$};
    \node[rotate=90,anchor=east] at (\rr,-0.5) {$\alt=$};
    %the filling
    \node at (\rr-1+0.5,5-0.5) {\poscross};
    \node at (\rr-2+0.5,8-0.5) {\negcross};
    \node at (\rr-3+0.5,4-0.5) {\fixcross};
  \end{tikzpicture}
  \caption{The symmetry of the evacuation diagram.}
  \label{fig:filling-from-evacuation-diagram}
\end{figure}
By the symmetry of the local rules the evacuation diagram for
$\ev \alt$ is obtained from the evacuation diagram for $\alt$ by
mirroring it along the diagonal and interchanging {\poscross} and
{\negcross}. The cell $(i,j)$ is interchanged with the cell
$(2r+1-j,2r+1-i)$.  This yields the part of
Theorem~\ref{thm:rotation-promotion-permutations} concerning
evacuation, see Figure~\ref{fig:filling-from-evacuation-diagram} for
an illustration.
\begin{thm} \label{thm:ev-rc}
  Let $\alt=(\emptyset = \wgt^0,\wgt^1,\dots,\wgt^{2r-1},\wgt^{2r}=\emptyset)$ be an alternating tableau. Suppose that $n\geq r$ if $n$ is even and $n\geq r-1$ if $n$ is odd. Let $\phi$ be the filling of the growth diagram $\G(\alt)$. Then the filling of $\G(\ev \alt)$ is obtained by rotating $\phi$ by $180\degree$.
\end{thm}
\begin{proof}
  Let $\phi_{\alt}=\phi$, respectively $\phi_{\ev\alt}$, be the
  fillings of the growth diagrams $\G(\alt)$, respectively
  $\G(\ev \alt)$.  Let $(i,j)$ be the position of a cell with a cross
  in the filling $\phi_{\alt}$. Then, according to
  Corollary~\ref{cor:relate-fillings}:
  \begin{enumerate}
      \item if $i<j$, there is a {\negcross} in the evacuation diagram of $\alt$ in $(2i-1,2j-1)$. Thus there is a {\poscross} in the evacuation diagram of $\ev\alt$ in $(2r-2j+2,2r-2i+2)$ and therefore there is a cross in $(r+1-i,r+1-j)$ in $\phi_{\ev\alt}$.
      \item if $j<i$, there is a {\poscross} in the evacuation diagram of $\alt$ in $(2j,2i)$. Thus there is a {\negcross} in the evacuation diagram of $\ev\alt$ in $(2r+1-2i,2r+1-2j)$ and therefore there is a cross in $(r+1-i,r+1-j)$ in $\phi_{\ev\alt}$.
      \item if $i=j$, then there is a {\fixcross} in the evacuation diagram of $\alt$ in $(2i-1,2i)$. Thus there is a {\fixcross} in the evacuation diagram of $\ev\alt$ in $(2r+1-2i,2r+2-2i)$ and therefore there is a cross in $(r+1-i,r+1-j)$ in $\phi_{\ev\alt}$.
  \end{enumerate}
\end{proof}

\begin{prop}\label{prop:rotate-filling}
    Consider the classical growth diagrams $\G$ and $\widetilde\G$ for the partial fillings $\phi$ and $\reco\phi$, where $\reco\phi$ is obtained by rotating $\phi$ by $180\degree$. Let $Q$ and $\widetilde Q$ (respectively $P$ and $\widetilde P$) be the partial standard Young tableaux corresponding to the sequence of partitions on the top borders (respectively right borders) of the growth diagrams $\G$ and $\widetilde \G$.  Then $\widetilde Q = \ev Q$ and $\widetilde P = \ev P$.
\end{prop}

We are now in the position to prove Theorem~\ref{thm:s1q-partial-permutations}, which we reformulate as follows.
\begin{thm}
  Let $\alt=(\emptyset = \wgt^0,\wgt^1,\dots,\wgt^{2r-1},\wgt^{2r}=\mu)$
  be an alternating tableau of length
  $r\leq \lfloor\frac{n+1}{2}\rfloor$. Let $\phi$ be the filling of
  the growth diagram $\G(\alt)$.
  Then the sequence of partitions on the bottom (respectively
  right) border of $\pos\G(\ev \alt)$ (respectively $\neg\G(\ev\alt)$) is obtained by evacuating the
  sequence of partitions on the bottom (respectively right) border of $\pos\G(\alt)$ (respectively $\neg\G(\alt)$).
  Moreover, the filling of $\G(\ev \alt)$ is obtained by rotating $\phi$ by $180\degree$.
\end{thm}

\begin{proof}
  We begin by extending $\alt$ to an alternating tableau of empty shape $\tilde\alt=(\emptyset=\tilde\wgt^0,\dots,\tilde\wgt^{2(r+r)}=\emptyset)$, such that $\tilde\wgt^i=\wgt^i$ for $i\leq r$, by appending the reversal of $\alt$. Let $\tilde\phi$ be the filling of $\G(\tilde\alt)$, which we divide into four parts, as illustrated in the left-most diagram below.
  Filling $A$ is the filling corresponding to $\alt$, filling $B$ is the part below and to the left of $\wgt^{2r}$, filling $C$ is the part above and to the right of $\wgt^{2r}$ and filling $D$ is the part below and to the right of $\wgt^{2r}$.

  By the symmetry of the local rules and the evacuation diagram as illustrated in Figure~\ref{fig:ev-adjoint-example} we see that $\ev\alt$ coincides with the first $2r+1$ staircases of $\pr^{(r)}(\ev\tilde\alt)$, where $\pr^{(r)}$ denotes $\underbrace{\pr\circ\pr\circ\dots\circ\pr}_{\text{$r$ times}}$.

  Let $Q$ be the sequence of partitions on the bottom border of $\pos\G(\alt)$. This sequence is also the sequence of partitions on the top border of the classical growth diagram with filling $B$.

    The inequality $r\leq \lfloor\frac{n+1}{2}\rfloor$ implies that $n\geq 2r$ if $n$ is even and $n\geq 2r-1$ if $n$ is odd. Applying Theorem~\ref{thm:ev-rc} and Theorem~\ref{thm:pr-rot} we obtain the following picture:
 
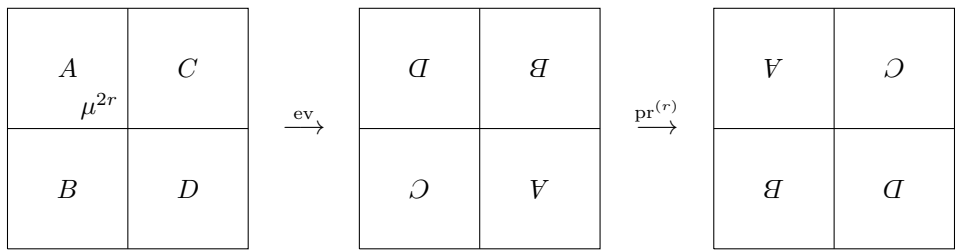
\[
  \begin{tikzpicture}[baseline={([yshift=-.5ex]current bounding box.center)},scale=0.8]
  \draw (0,0) -- (4,0) -- (4,4) -- (0,4) -- (0,0);
  \draw (2,0) -- (2,4);
  \draw (0,2) -- (4,2);
  \node at (1,3) {$A$};
  \node at (1,1) {$B$};
  \node at (3,3) {$C$};
  \node at (3,1) {$D$};
  \node[anchor=south east] at (2,2) {$\wgt^{2r}$};
  \end{tikzpicture} \quad\stackrel{\ev}{\longrightarrow}\quad
  \begin{tikzpicture}[baseline={([yshift=-.5ex]current bounding box.center)},scale=0.8]
  \draw (0,0) -- (4,0) -- (4,4) -- (0,4) -- (0,0);
  \draw (2,4) -- (2,0);
  \draw (4,2) -- (0,2);
  \node[rotate=180] at (3,1) {$A$};
  \node[rotate=180] at (3,3) {$B$};
  \node[rotate=180] at (1,1) {$C$};
  \node[rotate=180] at (1,3) {$D$};
  \end{tikzpicture} \quad\stackrel{\pr^{(r)}}{\longrightarrow}\quad
  \begin{tikzpicture}[baseline={([yshift=-.5ex]current bounding box.center)},scale=0.8]
  \draw (0,0) -- (4,0) -- (4,4) -- (0,4) -- (0,0);
  \draw (2,0) -- (2,4);
  \draw (0,2) -- (4,2);
  \node[rotate=180] at (1,3) {$A$};
  \node[rotate=180] at (1,1) {$B$};
  \node[rotate=180] at (3,3) {$C$};
  \node[rotate=180] at (3,1) {$D$};
  \end{tikzpicture}
  \]

  Thus the sequence of partitions on the bottom border of $\pos\G(\ev \alt)$ is the same as the sequence of partitions on the top border of the regular growth diagram with filling $\reco B$. By Proposition~\ref{prop:rotate-filling} we obtain the statement for the sequence of partitions on the bottom border.

  The result for the right border follows using the same argument, replacing the filling $A$ with the filling $C$.
\end{proof}

\subsection{\texorpdfstring{$\GL(2)$}{GL(2)}-alternating tableaux}
\label{sec:GL2}

To finish the proof of Theorem~\ref{thm:rotation-promotion-permutations}, it remains to consider the case $n=2$.

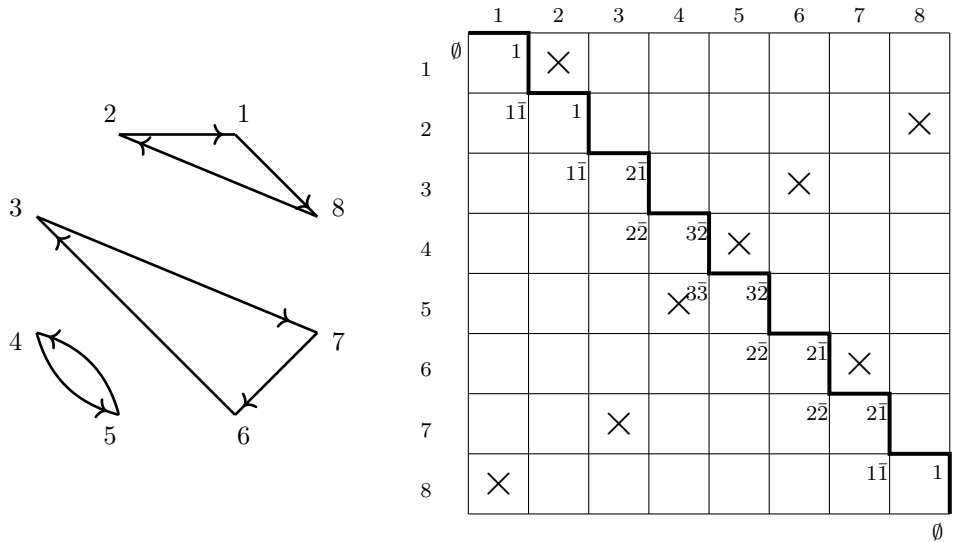
\begin{figure}
    \centering
    \tikzset{->-/.style={decoration={
        markings,
        mark=at position #1 with {\arrow{>}}},postaction={decorate}}}
    \raisebox{1.25cm}{%
  \begin{tikzpicture}[line width=1pt]
    \node (a) [draw=none, minimum size=4cm, regular polygon, regular polygon sides=8] at (0,0) {};
    \foreach \n [count=\i from 0, remember=\n as \lastn, evaluate={\i+\lastn}] in {1,2,...,8}
        \path (a.center) -- (a.corner \n) node[pos=1.15] {$\n$};
    \draw[->-=.9] (a.corner 1) to (a.corner 8);
    \draw[->-=.9] (a.corner 2) to (a.corner 1);
    \draw[->-=.9] (a.corner 3) to (a.corner 7);
    \draw[->-=.9, bend right] (a.corner 4) to (a.corner 5);
    \draw[->-=.9, bend right] (a.corner 5) to (a.corner 4);
    \draw[->-=.9] (a.corner 6) to (a.corner 3);
    \draw[->-=.9] (a.corner 7) to (a.corner 6);
    \draw[->-=.9] (a.corner 8) to (a.corner 2);
  \end{tikzpicture}}
  \hfill
  \begin{tikzpicture}[scale=0.8]\tiny
% thick diagonal
\draw[line width=1.5pt] (9,1) -- (9,2) -- (8,2) -- (8,3) -- (7,3) -- (7,4) -- (6,4) -- (6,5) -- (5,5) -- (5,6) -- (4,6) -- (4,7) -- (3,7) -- (3,8) -- (2,8) -- (2,9) -- (1,9);
\foreach \y in {1,...,9}
   \draw (1,\y)--(9,\y);
\foreach \x in {1,...,9}
   \draw (\x,1)--(\x,9);
% column labels
\foreach \x in {1,...,8}
   \draw (\x.5,9.3) node{\footnotesize$\x$};
% row labels
\foreach \x in {1,...,8}
   \draw (0.3,9.4-\x) node{\footnotesize$\x$};
% crosses
\draw (1.5,1.5) node{\huge$\times$};
\draw (2.5,8.5) node{\huge$\times$};
\draw (3.5,2.5) node{\huge$\times$};
\draw (4.5,4.5) node{\huge$\times$};
\draw (5.5,5.5) node{\huge$\times$};
\draw (6.5,6.5) node{\huge$\times$};
\draw (7.5,3.5) node{\huge$\times$};
\draw (8.5,7.5) node{\huge$\times$};
% staircases
\draw (0.8,8.7) node{\footnotesize$\emptyset$};
\draw (1.8,8.7) node{\footnotesize$1$};
\draw (1.8,7.7) node{\footnotesize$1\bar1$};
\draw (2.8,7.7) node{\footnotesize$1$};
\draw (2.8,6.7) node{\footnotesize$1\bar1$};
\draw (3.8,6.7) node{\footnotesize$2\bar1$};
\draw (3.8,5.7) node{\footnotesize$2\bar2$};
\draw (4.8,5.7) node{\footnotesize$3\bar2$};
\draw (4.8,4.7) node{\footnotesize$3\bar3$};
\draw (5.8,4.7) node{\footnotesize$3\bar2$};
\draw (5.8,3.7) node{\footnotesize$2\bar2$};
\draw (6.8,3.7) node{\footnotesize$2\bar1$};
\draw (6.8,2.7) node{\footnotesize$2\bar2$};
\draw (7.8,2.7) node{\footnotesize$2\bar1$};
\draw (7.8,1.7) node{\footnotesize$1\bar1$};
\draw (8.8,1.7) node{\footnotesize$1$};
\draw (8.8,0.7) node{\footnotesize$\emptyset$};
\end{tikzpicture}
\caption{A noncrossing set partition corresponding to a $\GL(2)$-alternating tableau.}
\label{fig:noncrossing-set-partition}
\end{figure}
\begin{lemma}\label{lem:noncrossing}
    The map $\Perm$ restricts to a bijection between $\GL(n)$-alternating tableaux of empty shape and length $r$, such that every staircase has at most two nonzero parts, and noncrossing set partitions on $\{1,\dots,r\}$.
\end{lemma}
\begin{proof}
    For simplicity, suppose that $\alt$ is a $\GL(2)$-alternating tableau. Let $\pi$ be the permutation corresponding to the filling associated with $\alt$. We show that, when drawn as a chord diagram as in Figure~\ref{fig:noncrossing-set-partition}, it is obtained from a noncrossing set partition by orienting the arcs delimiting the blocks clockwise, when the corners of the polygon are labelled counterclockwise.

    We say that two arcs $(i,\pi_i)$ and $(j,\pi_j)$ in the chord diagram, with $i<k$, \Dfn{cross}, if and only if the indices involved satisfy one of the following two inequalities:
    \[
        i<j\leq \pi_i<\pi_j \qquad\text{or}\qquad \pi_i<\pi_j<i<j.
    \]
    Let us remark that this is precisely Corteel's~\cite{MR2290808} notion of crossing in permutations.

    It follows by direct inspection that the chord diagram corresponds to a noncrossing partition in the sense above if and only if no two arcs cross.

    Moreover, a crossing of the first kind is the same as a pair of crosses in the rectangle below and to the left of the cell in row and column $j$ of $\G(\alt)$, such that one cross is above and to the left of the other.  Similarly, a crossing of the second kind is the same as a pair of crosses in the rectangle above and to the right of the cell in row and column $i$ of $\G(\alt)$, such that one cross is above and to the left of the other.

    By construction, a $\GL(2)$-alternating tableau cannot contain a vector with both entries strictly positive or both entries strictly negative. Thus, such pairs of crosses may not occur.
\end{proof}

We can now prove another part of Theorem~\ref{thm:rotation-promotion-permutations}.
\begin{thm}
    Let $n\leq 2$ and let $\alt$ be a $\GL(n)$-alternating tableau of empty shape. Then $\rot\Perm(\alt)=\Perm(\pr\alt)$.
\end{thm}
\begin{proof}
    Let $r$ be the length of $\alt$ and let $\hat\alt$ be the $\GL(r)$-alternating tableau obtained from $\alt$ by inserting $r-n$ zeros into each staircase. Then, by Theorem~\ref{thm:pr-rot}, $\Perm(\pr\hat\alt)=\rot\Perm(\hat\alt)$. By Lemma~\ref{lem:noncrossing}, the staircases in the alternating tableau corresponding to $\rot\Perm(\hat\alt)$ have at most two nonzero parts. Thus, the claim follows from Theorem~\ref{thm:stability}.
\end{proof}

To finish the proof of Theorem~\ref{thm:rotation-promotion-permutations}, we show that the evacuation of a $\GL(2)$-alternating tableaux of empty shape is just its reversal.
\begin{thm}
    Let $\alt$ be a $\GL(2)$-alternating tableau of empty shape. Then $\ev\alt$ is the reversal of $\alt$.
\end{thm}
\begin{proof}
    Let $\alt=(\emptyset=\wgt^0,\dots,\wgt^{2r}=\emptyset)$ and let $\ev\alt=\alt=(\emptyset=\tilde\wgt^0,\dots,\tilde\wgt^{2r}=\emptyset)$. Note that $\tilde\wgt^{2i}$ is the $2i$-th (counting from zero) staircase in $\pr^{(r-i)}\alt$.  Thus, its negative part is the same as the negation of the positive part of $\wgt^{2(r-i)}$, because the fillings in the respective regions of the corresponding growth diagrams coincide. Because the negative part and the positive part of the even labelled staircases of a $\GL(2)$-alternating tableau are equal, we conclude that $\wgt^{2(r-i)}=\tilde\wgt^{2i}$.

    It remains to show that $\wgt^{2(r-i)-1}=\tilde\wgt^{2i+1}$. If $\tilde\wgt^{2i}\neq\tilde\wgt^{2(i+1)}$, the staircase $\tilde\wgt^{2i+1}$ is uniquely determined. Otherwise, if $\tilde\wgt^{2i}=\tilde\wgt^{2(i+1)}$, it is obtained from $\tilde\wgt^{2i}$ by adding the unit vector $e_1$ if and only if $i$ is a fixed point of $\Perm(\tilde\alt)$. Equivalently, this is the case if and only if $r+1-i$ is a fixed point of $\Perm(\alt)$, as can be seen by inspecting the evacuation diagram.
\end{proof}

\subsection{Promotion and evacuation of oscillating tableaux}
\label{sec:oscillating}
We now deduce Theorem~\ref{thm:s1q-partial-matchings} and Theorem~\ref{thm:rotation-promotion-matchings} from the results in the preceding section, by demonstrating that oscillating tableaux can be regarded as special alternating tableaux.

For two partitions $\lambda,\mu$ we define
\begin{align*}
    \lambda \vee \mu &:= \max(\lambda,\mu)\\
    \lambda \wedge \mu &:= \min(\lambda,\mu)
\end{align*} where $\max$ and $\min$ are defined componentwise.

Consider an $n$-symplectic oscillating tableau
$\osc=(\wgtosc^0,\wgtosc^1,\dots,\wgtosc^r)$.  Then
\[
\alt_\osc=[\wgtosc^0,\wgtosc^0]_n, [\wgtosc^0 \vee \wgtosc^1, \wgtosc^0 \wedge \wgtosc^1]_n, [\wgtosc^1,\wgtosc^1]_n, \ldots , [ \wgtosc^{r-1} \vee \wgtosc^r, \wgtosc^{r-1} \wedge \wgtosc^r]_n, [\wgtosc^r,\wgtosc^r]_n
\]
is a $\GL(n)$-alternating tableau.  Because $\wgtosc^i$ and
$\wgtosc^{i+1}$ differ by a unit vector, the staircase
$[\wgtosc^i \vee \wgtosc^{i+1}, \wgtosc^i \wedge \wgtosc^{i+1}]_n$ is
obtained by taking the larger partition as positive part, and the
smaller partition as negative part.

If $\osc$ is an oscillating tableau, the filling of $\G(\alt_\osc)$
is symmetric with respect to the diagonal from the top-left to the
bottom-right.  In particular, if $\osc$ has empty shape the filling
is precisely the permutation obtained by interpreting the perfect
matching as a fixed point free involution.

Conversely, suppose that $\alt$ is an alternating tableau such that
the filling of $\G(\alt)$ is symmetric with respect to the diagonal
from the top-left to the bottom-right, and has no crosses on this
diagonal.  Then, taking the positive part of every second staircase
in $\alt$ we obtain an oscillating tableau $\osc_\alt$.  The filling
of $\G(\osc_\alt)$ is precisely the part of $\G(\alt)$ below and to
the left of the diagonal.

It is easy to see that the rotation of a fixed point free involution corresponds to the rotation of the associated perfect matching.  Also, the reversal of the complement of a symmetric filling corresponds to the reversal of the associated perfect matching.  Thus, it remains to show that this correspondence between oscillating tableaux and certain alternating tableaux intertwines promotion of oscillating tableaux and alternating tableaux: $\osc_{\pr\alt} = \pr\osc_\alt$.

\begin{lemma}
  The promotion of an oscillating tableau equals the oscillating tableau corresponding to the promotion of the associated alternating tableau: $\osc_{\pr\alt} = \pr\osc_\alt$.
\end{lemma}
\begin{proof}
  Let $\osc=(\emptyset=\wgtosc^0,\dots,\wgtosc^r=\emptyset)$ be an
  oscillating tableau.  By Theorem~\ref{thm:local_commutor} its
  promotion
  $\pr\osc=(\emptyset=\half\wgtosc^0,\dots,\half\wgtosc^r=\emptyset)$
  can be computed using the local rule from
  Definition~\ref{def:local-rules}:
  \begin{equation}
    \label{eq:local-rule-oscillating}
    \half\wgtosc^{i-1} = %
    \dom_{\fH_n}(\half\wgtosc^{i-2}+\wgtosc^i-\wgtosc^{i-1}),
  \end{equation}
  where $\fH_n$ is the hyperoctahedral group, the Weyl group of the
  symplectic group $\Sp(2n)$.  Recall that in this case the dominant
  representative of a vector is obtained by sorting the absolute
  values of its components into decreasing order.

  Let $\alt_\osc=(\emptyset=\wgt^0,\dots,\wgt^{2r}=\emptyset)$ be
  the alternating tableau associated with the oscillating tableau
  $\osc$.  Let
  $\halfpr\alt_\osc=(\emptyset =
  \hm^0,\hm^1,\dots,\hm^{2r-1},\hm^{2r}=\emptyset)$ be as in
  Definition~\ref{def:half-promotion} and let
  $\pr\alt_\osc = (\emptyset = \Pm^0,\dots,\Pm^{2r}=\emptyset)$ be the
  promotion of $\alt_\osc$.

  We have to show that for every square in the promotion diagram of
  the alternating tableau
  \begin{equation}
    \label{eq:square}
    \def\arraystretch{2.5}
    \begin{array}{ccc}
        \wgt^{2i-2} & \wgt^{2i-1}  & \wgt^{2i}\\
        \hm^{2i-3} & \hm^{2i-2}  & \hm^{2i-1}\\
        \Pm^{2i-4} & \Pm^{2i-3}  & \Pm^{2i-2}
    \end{array}
  \end{equation}
  the positive parts of the four corners
  $\wgtosc^{i-1} = \pos{\wgt^{2i-2}}$,
  $\wgtosc^{i} = \pos{\wgt^{2i}}$,
  $\half{\wgtosc}^{i-2} = \pos{\Pm^{2i-4}}$ and
  $\half{\wgtosc}^{i-1} = \pos{\Pm^{2i-2}}$ satisfy
  Equation~\eqref{eq:local-rule-oscillating}.  Note that the positive
  parts and the negative parts of these staircases coincide.  To
  avoid superscripts, we set $\wgt^{2i-2}=[\lambda,\lambda]_n$,
  $\wgt^{2i}=[\nu,\nu]_n$, $\Pm^{2i-4}=[\kappa,\kappa]_n$ and
  $\Pm^{2i-2}=[\mu,\mu]_n$.

  Because the filling of $\G(\alt_\osc)$ is symmetric, the position
  $\ell$ of the cross in the first row equals the position $k$ of the
  cross in the first column.  Thus, we have $\ell=k$ in
  Lemma~\ref{lem:alternating-equalities}.  Let us consider the case
  $i\neq\ell$ first.  We assume that $i<\ell$, the case of $i>\ell$
  is very similar. If $i<\ell$, that is, $2i\leq 2\ell-2$, the
  positive parts of staircases in the same column of the first two
  rows of diagram~\eqref{eq:square} coincide by
  Lemma~\ref{lem:alternating-equalities}(a).  By
  Lemma~\ref{lem:alternating-equalities}(a'), the negative parts of
  staircases in the same column of the second two rows coincide.

  Moreover, by construction of $\alt_\osc$, the staircase in the
  middle of the first row either equals $[\lambda, \nu]_n$ or
  $[\nu, \lambda]_n$.  Let us assume the latter, the former case is
  dealt with similarly.  For the staircase in the middle we then
  obtain, applying the local rule to the staircases on the top-left,
  \[
    \hm^{2i-2}=\dom_{\fS_n}\big([\lambda,\kappa]_n+[\nu,\lambda]_n-[\lambda,\lambda]_n\big)
    =[\nu,\kappa]_n.
  \]
  Similarly, applying the local rule to the staircases on the
  bottom-right, we find
  \[
    [\mu,\mu]_n=\dom_{\fS_n}\big([\pos{\Pm^{2i-3}},\kappa]_n+[\nu,\mu]_n-[\nu,\kappa]_n\big)\\
    =[\pos{\Pm^{2i-3}},\mu]_n.
  \]
  Therefore, the square of staircases in diagram~\eqref{eq:square}
  has the following form:
  \begin{equation}
    \label{eq:square-special}
    \def\arraystretch{2.5}
    \begin{array}{ccc}
        [\lambda,\lambda]_n & [\nu,\lambda]_n & [\nu,\nu]_n\\\relax
        [\lambda,\kappa]_n & [\nu,\kappa]_n & [\nu,\mu]_n\\\relax
        [\kappa,\kappa]_n & [\mu,\kappa]_n & [\mu,\mu]_n
    \end{array}
  \end{equation}
  Because the negative parts of the four staircases in the lower left
  corner are all the same, the positive parts satisfy
  $\mu=\dom_{\fS_n}(\kappa+\nu-\lambda)$, and therefore also
  Equation~\eqref{eq:local-rule-oscillating}.

  It remains to show that Equation~\eqref{eq:local-rule-oscillating}
  also holds for $i=\ell$.  By
  Lemma~\ref{lem:alternating-equalities}(c), the positive parts of the
  staircases $\wgt^{2\ell-2}$, $\wgt^{2\ell-1}$ and $\hm^{2\ell-3}$
  all coincide, and thus equal $\lambda$.  Moreover, the positive
  part $\alpha$ of the staircase in the middle is obtained by adding
  a cell to the first column of the Ferrers diagram of $\lambda$.

  By Lemma~\ref{lem:alternating-equalities}(c'), the positive parts
  of the staircases $\hm^{2\ell-1}$, $\Pm^{2\ell-3}$ and
  $\Pm^{2\ell-2}$ all coincide, and thus equal $\mu$.  Moreover, the
  positive part $\alpha$ of the staircase in the middle is obtained
  by adding a cell to the first column of the Ferrers diagram of
  $\mu$.  Therefore $\lambda=\mu$.

  Lemma~\ref{lem:alternating-equalities}(b) implies that the negative
  parts of the staircases $\wgt^{2\ell-1}$, $\wgt^{2\ell}$,
  $\hm^{2\ell-2}$ and $\hm^{2\ell-1}$ are all equal to $\nu$.
  Finally, Lemma~\ref{lem:alternating-equalities}(a') shows that the
  negative parts of the staircases $\hm^{2\ell-3}$, $\hm^{2\ell-2}$,
  $\Pm^{2\ell-4}$ and $\Pm^{2\ell-3}$ are all equal to $\kappa$.
  Thus $\nu=\kappa$ and diagram~\eqref{eq:square} has the form
  \begin{equation}
    \label{eq:square-special-k=l}
    \def\arraystretch{2.5}
    \begin{array}{ccc}
        [\lambda,\lambda]_n & [\lambda,\nu]_n & [\nu,\nu]_n\\\relax
        [\lambda,\nu]_n & [\alpha,\nu]_n & [\lambda,\nu]_n\\\relax
        [\nu,\nu]_n & [\lambda,\nu]_n & [\lambda,\lambda]_n
    \end{array}
  \end{equation}
  Considering the growth diagram $\G(\alt_\osc)$, we additionally
  find that $\lambda$ is obtained from $\nu$ by adding a cell to the
  first column. %, essentially because $\ell=k$.
  Thus, the vector $\nu + \nu - \lambda$ is obtained from $\nu$ by
  subtracting $1$ from the entry at position $\ell(\lambda)$, which
  is $0$ in $\nu$.  Taking the absolute values of the entries of the
  vector $\nu + \nu - \lambda$ then yields $\lambda$.
\end{proof}

\printbibliography
\end{document}